\documentclass[11pt]{amsart}
\usepackage{amsthm,amsfonts,amssymb,amsmath,oldgerm,mathabx}
\numberwithin{equation}{section}
\usepackage{fullpage}
\usepackage{color}

\usepackage[pdftex]{graphicx} 
\usepackage{hyperref}
\usepackage{stmaryrd}
\usepackage{comment}

\usepackage{caption}
\usepackage{subcaption}
\usepackage{cases}

\hypersetup{
    colorlinks,           
    citecolor=blue,    
    filecolor=blue,     
    urlcolor=gray,      
    linkcolor=blue,
}

\DeclareMathOperator\eD{e}

\def\eps{\varepsilon }
\DeclareMathOperator{\Id}{Id}

\DeclareMathOperator{\Tr}{Tr}
\DeclareMathOperator{\dD}{d}

\DeclareMathOperator{\Res}{Res}

\newcommand\br{\begin{remark}}
\newcommand\er{\end{remark}}
\newcommand\bp{\begin{pmatrix}}
\newcommand\ep{\end{pmatrix}}
\newcommand{\be}{\begin{equation}}
\newcommand{\ee}{\end{equation}}
\newcommand{\bes}{\begin{equation*}}
\newcommand{\ees}{\end{equation*}}
\newcommand\ba{\begin{equation}\begin{aligned}}
\newcommand\ea{\end{aligned}\end{equation}}
\newcommand\bas{\begin{equation*}\begin{aligned}}
\newcommand\eas{\end{aligned}\end{equation*}}
\newcommand\ds{\displaystyle}

\newcommand{\beg}{\begin{example}}
\newcommand{\eeg}{\end{example}}
\newcommand{\bpr}{\begin{proposition}}
\newcommand{\epr}{\end{proposition}}
\newcommand{\bt}{\begin{theorem}}
\newcommand{\et}{\end{theorem}}
\newcommand{\bc}{\begin{corollary}}
\newcommand{\ec}{\end{corollary}}
\newcommand{\bl}{\begin{lemma}}
\newcommand{\el}{\end{lemma}}
\newcommand{\bd}{\begin{definition}}
\newcommand{\ed}{\end{definition}}
\newcommand{\brs}{\begin{remarks}}
\newcommand{\ers}{\end{remarks}}

\newtheorem{theorem}{Theorem}[section]
\newtheorem{proposition}[theorem]{Proposition}
\newtheorem{corollary}[theorem]{Corollary}
\newtheorem{lemma}[theorem]{Lemma}

\theoremstyle{remark}
\newtheorem{remark}[theorem]{Remark}
\theoremstyle{definition}
\newtheorem{definition}[theorem]{Definition}

\newtheorem{example}[theorem]{Example}

\newcommand\R{\mathbf R}

\newcommand{\N}{\mathbf N}

\newcommand\bB{{\mathbf B}}

\newcommand\bE{{\mathbf E}}
\newcommand\bF{{\mathbf F}}
\newcommand\bG{{\mathbf G}}

\newcommand\bJ{{\mathbf J}}

\newcommand\bL{{\mathbf L}}

\newcommand\bT{{\mathbf T}}

\newcommand\bfe{{\mathbf e}}

\newcommand\bfu{{\mathbf u}}
\newcommand\bfv{{\mathbf v}}
\newcommand\bfw{{\mathbf w}}
\newcommand\bfx{{\mathbf x}}
\newcommand\bfy{{\mathbf y}}
\newcommand\bfz{{\mathbf z}}

\newcommand\ube{{\underline \bfe}}

\newcommand\ubw{{\underline \bfw}}
\newcommand\ubx{{\underline \bfx}}
\newcommand\uby{{\underline \bfy}}
\newcommand\ubz{{\underline \bfz}}

\newcommand\te{\widetilde{e}}

\newcommand\tg{\widetilde{g}}

\newcommand\ttau{\widetilde{\tau}}

\newcommand\tbw{\widetilde{\bfw}}
\newcommand\tbx{\widetilde{\bfx}}
\newcommand\tby{\widetilde{\bfy}}

\newcommand\ua{{\underline a}}
\newcommand\ub{{\underline b}}

\newcommand\ue{{\underline e}}

\newcommand\ug{{\underline g}}

\newcommand\cB{{\mathcal B}}

\newcommand\cL{{\mathcal L}}

\newcommand\cO{{\mathcal O}}

\newcommand\cT{{\mathcal T}}

\newcommand\cV{{\mathcal V}}

\usepackage{enumitem}

\newcommand{\GC}{\bfx_{\rm gc}}
\newcommand{\eGC}{e_{\rm gc}}
\newcommand{\uGC}{\underline{\GC}}
\newcommand{\ueGC}{\underline{\eGC}}
\newcommand{\tGC}{\tbx_{\rm gc}}
\newcommand{\teGC}{\te_{\rm gc}}
\newcommand{\yGC}{\bfy_{\rm gc}}
\newcommand{\gGC}{g_{\rm gc}}
\newcommand{\tyGC}{\tby_{\rm gc}}
\newcommand{\tgGC}{\tg_{\rm gc}}

\title[Analysis of a Crank-Nicolson scheme for plasmas]{Convergence analysis of a Crank-Nicolson scheme\\
for strongly magnetized plasmas}

\author{Francis Filbet}
\address{Universit\'e de Toulouse,  Institut de Math\'ematiques de
  Toulouse, F-31062 cedex, France}
\email{francis.filbet@math.univ-toulouse.fr}
\thanks{}

\author{L.~Miguel Rodrigues}
\address{Universit\'e de Rennes, Institut de Recherche en Math\'ematiques
  de Rennes, F-35000 Rennes, France}
\email{luis-miguel.rodrigues@univ-rennes.fr}
\thanks{Research of L.M.R. was partially supported by the ANR Project HEAD ANR-24-CE40-3260 and the Institut Universitaire de France.}

\author{Kim Han Trinh}
\address{Universit\'e de  Rennes, Institut de Recherche en Math\'ematiques
  de Rennes, F-35000 Rennes, France}
\email{kim-han.trinh@univ-rennes.fr}

\begin{document}

\begin{abstract}

The present paper is devoted to the convergence analysis of an asymptotic preserving particle scheme designed to serve as a particle pusher in a Particle-In-Cell (PIC) method for the Vlasov equation with a strong inhomogeneous magnetic field. The asymptotic preserving scheme that we study removes classical strong restrictive stability constraints on discretization steps while capturing the large-scale dynamics, even when the discretization is too coarse to capture fastest scales. Our error bounds are explicit regarding the discretization and stiffness parameters and match sharply numerical tests. The present analysis is expected to be representative of the general analysis of a class of schemes, developed by the authors, conceived as implicit-explicit schemes on augmented formulations.

\vspace{0.5em}


\vspace{0.5em}

\end{abstract}

\date{\today}
\maketitle

\tableofcontents

\section{Introduction}\label{s:introduction}

The numerical analysis of the present paper is motivated by the numerical computations of the dynamics of plasmas where the charged particles evolve under an electrostatic  and intense confining magnetic field. This configuration is typical of a tokamak plasma \cite{bellan_2006_fundamentals, miyamoto_2006_plasma} where the magnetic field is used to confine particles inside the core of the device. Kinetic models, based on a mesoscopic description of the various particles constituting a plasma, and coupled to Maxwell's equations for the computation of the electromagnetic fields, are very precise approaches for the study of such thermonuclear fusion plasmas. Incidentally we point out that magnetized plasmas are also encountered in a wide variety of astrophysical situations, and similar kinetic/Maxwell models, thus also our associated analysis, are relevant for those too.

Our present analysis contributes to the convergence analysis of numerical approximations obtained by particle methods (see \cite{BiLa85}), which consist in approximating the particle distribution function by a finite number of macro-particles. The trajectories of these particles are determined from the characteristic curves associated  to the underlying Vlasov equation; see \eqref{eq:ODE} below for the characteristic differential equations studied here. From the presence of a strong magnetic field stems a stiff /dynamics where high-frequency oscillations co-exist with slower dynamics. Insisting on capturing the full stiff dynamics introduces a stringent restriction on discretization meshes. Over the last decade a wealth of energy has been poured into designing schemes relaxing this constraint. Let us simply refer the reader to \cite{FRT25} for a detailed discussion and comparison of various particle pushers for PIC methods and to the introduction of \cite{FiRo23} for an account of other families of methods. In contrast, much fewer effort has been devoted to provide a complete mathematical analysis of the designed schemes. We mention \cite{CCLMZ19,HLW20,hairer2022,WZ21,WJ23,Yin25} as reprensative notable exceptions. It is worth pointing out that, in contrast with our present analysis, inspired by \cite{FiRo21}, all of those analyses fit in the frame of and use tools from the classical theory of geometric numerical integration, as described in \cite{HLW-book}.

Before delving into concrete description and analysis let us replace the scheme we propose and completely analyze here into the family of schemes \cite{FiRo16,FiRo17,FiRo23,FRT25} it belongs to. The starting point is the hope that if one aims at simply capturing accurately the slow part of the dynamics the stiffness step constraint could be relaxed and the observation that in simplest cases this could be achieved through suitable semi-implicit (IMEX) particles pushers. In \cite{FiRo16} this was carried out to capture the spatial dynamics in uniform magnetic fields. A complete analysis of this situation is provided in \cite{FiRo21}. A genuinely new idea was needed to extend the approach to realistic situations when one cannot focus simply on the spatial dynamics. The solution proposed in \cite{FiRo17}, and explained further below, is to apply the approach of \cite{FiRo16} to an extended system. In \cite{FiRo17} this is achieved for inhomogeneous magnetic fields with a fixed direction. In \cite{FiRo23} the strategy is adapted to more realistic geometries and in \cite{FRT25} it is shown how to improve conservation properties. Our present goal is to provide a complete analysis for schemes of the kind designed in \cite{FRT25}, thus the present contribution is thought to be to \cite{FRT25} what \cite{FiRo21} is to \cite{FiRo16}. The scheme we completely analyze is very close to the ones introduced and tested in \cite{FRT25} but we propose a new scheme so as to minimize the technicalities of the analysis. We restrict here to proving error bounds for the computation of characteristics but we refer to \cite{FiRo21} for a discussion on how this may be converted into errors on particle distribution functions.

Considering time-independent external electromagnetic fields with a fixed-direction magnetic field and an electric field deriving from an electric potential and focusing on the transverse dynamics on a time scale adapted to the strength of the magnetic field leads to the following set of characteristic differential system
\begin{align} \label{eq:ODE}
    \begin{cases}
        & \eps \dfrac{{\dD} \bfx}{{\dD} t} \, = \, \bfv \,, \\[0.5em]
        & \eps \dfrac{{\dD} \bfv}{{\dD} t} \,=\, \bE(\bfx) \,- \, b( \bfx) \dfrac{\bfv^{\perp}}{\eps} \,,
    \end{cases}
\end{align}
where the transverse space-velocity unknown $(\bfx,\bfv)$ is valued in $\R^2\times\R^2$, $\eps>0$ encodes the typical strength of the magnetic field, $b$ tracks its spatial variation and is lower bounded away from zero
\[
b(\bfx)\geq b_0>0\,,
\]
$\bE$ denotes the electric field and is assumed to take the form 
\[
\bE(\bfx) \,=\, -\, \nabla_{\bfx} \phi(\bfx)
\]
for some electric potential $\phi$. Both $b$ and $\phi$ and all their derivatives are assumed to be smooth and bounded. The symbol ${}^\perp$ denotes the direct orthogonal rotation, $\bfw^{\perp}=(-w_2,w_1)$. System \eqref{eq:ODE} is meant to be solved on a time interval $[0,T]$ and completed by an initial condition $\bfx(0) \,=\, \bfx^{0}$, $\bfv(0) \,=\, \bfv^{0}$. Note that solutions to \eqref{eq:ODE} also satisfy the energy conservation
\[
\frac{\dD}{\dD t}\left(\,\frac12\|\bfv\|^2+\phi(\bfx)\,\right)\,=\,0\,.
\] 

Since the arguments are to be generalized at the discrete level we provide in Section~\ref{s:continuous} a proof of the quite well-known fact, stated here as Proposition~\ref{p:cont-x}, that for any solution $(\bfx,\bfv)$ to \eqref{eq:ODE}, $(\bfx,\tfrac12\|\bfv\|^2)$ remains $\eps$-close to a solution $(\bfy,g)$ to
\begin{align} \label{eq:asymp}
    \begin{cases}
        & \dfrac{{\dD} \bfy}{{\dD} t} \, = \,-\dfrac{\bE^\perp}{b}(\bfy) \,-\,g\,\nabla^\perp_{\bfx}\left(\dfrac1b\right)(\bfy)\,, \\[0.5em]
        & \dfrac{{\dD} g}{{\dD} t} \, = \,-\dfrac{{\dD}\,\phi(\bfy)}{{\dD} t} \,. 
    \end{cases}
\end{align}
When doing so we follow closely the approach of \cite{FiRo20}. The variables $(\bfx,\tfrac12\|\bfv\|^2)$ are precisely the slow variables that we want to keep computing accurately even in situations when $\bfv$ is oscillating too fast to be reasonably computed. Before moving to the discrete level, let us stress that something that is difficult to enforce at the discrete level is that at the continuous level, as $\eps$ goes to zero, the velocity $\bfv$ converges weakly to zero whereas the kinetic energy $\tfrac12\|\bfv\|^2$ converges to a nonzero $g$ compatible with energy conservation, both being needed in the derivation of \eqref{eq:asymp}. 

The accurate computation of $(\bfx,\tfrac12\|\bfv\|^2)$ shall be enforced through the requirement that the discretization should be compatible with $\epsilon$ asymptotics. This is the asymptotic-preserving paradigm. For recent overviews on the latter we refer to \cite{DeDe17,Jin_APnew}. To achieve this, the relatively simple strategy initiated in \cite{FiRo17} is to compute solutions to \eqref{eq:asymp} as special solutions of a larger system where $(\bfv,\tfrac12\|\bfv\|^2)$ are artificially separated, say as $(\bfw,e)$ so that numerical schemes are allowed to damp $\bfw$ to zero while letting $e$ converge to a suitable $g$ as $\eps\to0$. Note that the point is that the point is that the continuous time evolution preserves for the augmented vector $(x,e,\bfw)$ preserves the relation $e=\tfrac12\|\bfw\|^2$ whereas the corresponding discretized dynamics does not. Thus picking a numerical scheme for \eqref{eq:ODE} along the strategy of \cite{FiRo17,FiRo23,FRT25} consists in choosing
\begin{itemize}
\item an augmented formulation for $(\bfx,e,\bfw)$, reducing to \eqref{eq:ODE} with $\bfv=\bfw$ when $e=\tfrac12\|\bfw\|^2$ 
\item a discretization of the augmented formulation compatible with $\eps$ asymptotics by \eqref{eq:asymp}.
\end{itemize}
Let us stress that the extra cost of considering an augmented formulation is considered as affordable here because it is to be incorporated in a particle method. It is far from obvious that a similar strategy could be incorporated in different numerical methods. 

The augmented system we choose here is 
\begin{align} \label{eq:aug}
    \begin{cases}
        & \eps \dfrac{{\dD} \bfx}{{\dD} t} \, = \, \bfw \,-\,\eps\,\left(e-\frac12\|\bfw\|^2\right)\,\nabla^\perp_{\bfx}\left(\dfrac1b\right)(\bfx)\,, \\[0.5em]
        & \eps \dfrac{{\dD} e}{{\dD} t} \, = \,-\eps \dfrac{{\dD}\,\phi(\bfx)}{{\dD} t} \,, \\[0.5em]
        & \eps \dfrac{{\dD} \bfw}{{\dD} t} \,=\, \bE(\bfx) \,- \, b( \bfx) \dfrac{\bfw^{\perp}}{\eps}\,.
    \end{cases}
\end{align}
Now consider a time step $\Delta t>0$ and define $(\bfx^n,e^n,\bfw^n)$, to serve as an approximation of $(\bfx,e,\bfw)(n\,\Delta t)$ when $(\bfx,e,\bfw)$ solves \eqref{eq:aug}, through the scheme 
\begin{equation} \label{scheme:CN}
\begin{cases}
        & \dfrac{\bfx^{n+1} - \bfx^{n}}{\Delta t} \, =\, \dfrac{\ubw^{n+1/2}}{\eps} 
        \,-\,\left(\ue^{n+1/2} \,-\, \dfrac{1}{2} \| \ubw^{n+1/2} \|^{2}\right)\nabla_{\bfx}^{\perp}\left(\dfrac1b\right)(\ubx^{n+1/2})\,, \\[0.5em]
        & \dfrac{e^{n+1} - e^{n}}{\Delta t} \,= \, \dfrac{\phi(\bfx^{n}) - \phi(\bfx^{n+1})}{\Delta t} \,,\\[0.5em]
        & \dfrac{\bfw^{n+1} - \bfw^{n}}{\Delta t} \,= \, \dfrac{1}{\eps} \bE(\ubx^{n+1/2}) \, - \, b(\ubx^{n+1/2}) \dfrac{(\ubw^{n+1/2})^{\perp}}{\eps^{2}} \,, 
\end{cases}
\end{equation}
where the underlining at half times denotes averaging
\begin{align*}
\ubz^{n+1/2} \,:=\, \dfrac{\bfz^{n+1} + \bfz^{n}}{2}\,,
\end{align*}
and initialization is $\bfx^{0}\, =\,\bfx^{0}$, $\bfw^{0} \,=\, \bfv^0$, $e^{0} \,=\,\tfrac12\|\bfv^0\|^2$.  
Note that by design \eqref{scheme:CN} does not propagate the relation $e^n=\tfrac12\|\bfw^n\|$ that holds initially. In Appendix~\ref{s:implicit} we check that Scheme~\eqref{scheme:CN} possesses a unique solution when $\Delta t$ is sufficiently small, independently of $\eps$.

The scheme we consider here differs from the one studied in \cite{FRT25} mainly on the choice of the augmented system. Explicitly we have decided to rely on energy conservation to describe the kinetic energy evolution and to introduce in the evolution equation of $\bfx$ the extra "ghost" term --- here $-\,\eps\,\left(e-\frac12\|\bfw\|^2\right)\,\nabla^\perp_{\bfx}\left(b^{-1}\right)(\bfx)$ --- that vanishes identically at the continuous level for the kind of initial data we are interested in but contributes to enforce the right $\eps$ asymptotics at the discrete level. In contrast, in the augmented system of \cite{FRT25} the dynamics of $(\bfx,e)$ is chosen to be
\begin{align*}
\eps \dfrac{{\dD} \bfx}{{\dD} t} &\, = \, \bfw\,,&
\eps \dfrac{{\dD} e}{{\dD} t}&\,=\,\bE(\bfx)\cdot\bfw\,,
\end{align*}
and a "ghost" term is introduced in the evolution equation of $\bfw$. Our numerical tests show similar performances between the scheme of \cite{FRT25} and the one studied here. The complete convergence study of the scheme of \cite{FRT25} would require a similar but lengthier analysis.

If we were to couple the underlying Vlasov equation with field equations we would need to reconstruct the velocity $\bfv^{n}$ from $(e^n,\bfw^n)$. A natural choice, compatible with energy conservation, is
\[
\bfv^{n} \,= \,\sqrt{2\,e^{n}} \dfrac{\bfw^{n}}{\|\bfw^{n}\|} \,,
\]
when $e^n$ is nonnegative and $\bfw^n$ is non zero. 

Even at the formal level it is not completely obvious that System~\eqref{eq:asymp} does describe the $\eps$ asymptotics of the slow part of \eqref{eq:ODE}. In contrast at fixed $\Delta t$ one may easily guess what could be the asymptotic scheme for the slow part of solutions to \eqref{scheme:CN} that do not blow up when $\eps\to0$. Indeed one may expect that at fixed $\Delta t$, for such solutions, $\ubw^{n+1/2}\eps^{-1}+(\bE^\perp/b)(\ubx^{n+1/2})$ becomes $\eps$-small as $\eps$ goes to zero so that $(\bfx^n,e^n)$ becomes $\eps$-close to $(\bfy^n,g^n)$ solving
\begin{equation} \label{scheme:asymp}
\begin{cases}
        & \dfrac{\bfy^{n+1} - \bfy^{n}}{\Delta t} \,=\, - \dfrac{\bE^{\perp}}{b}(\uby^{n+1/2})
        \,-\,\ug^{n+1/2}\nabla_{\bfx}^{\perp}\left(\dfrac1b\right)(\uby^{n+1/2})\,, \\
        & \dfrac{g^{n+1} - g^{n}}{\Delta t} \, = \,-\dfrac{\phi(\bfy^{n+1}) - \phi(\bfy^{n})}{\Delta t} \,.
\end{cases}
\end{equation}
Note that \eqref{scheme:CN} and \eqref{scheme:asymp} are second-order schemes respectively for \eqref{eq:aug} and \eqref{eq:asymp}. Such simple first guarantees are already provided in \cite{FiRo16,FiRo17,FiRo23,FRT25} and are typically named there asymptotic consistency as $\eps\to0$ for a fixed $\Delta t$. Our present goal is to prove that indeed solutions of \eqref{scheme:CN} do not blow up --- inherently a stability question --- and to provide uniform estimates when $(\eps,\Delta t)$ is small but no relation between $\Delta t$ and (any power of) $\eps$ is assumed.

\subsection*{Main results}

Our main result is the following theorem that quantifies in which sense $\left(\bfx,\tfrac12\|\bfv\|^2\right)$ is indeed well-captured even when the time step $\Delta t$ is too large for oscillations at frequency $\eps^{-2}$.

\begin{theorem}\label{th:main}
Let $T>0$, $M>0$ and $\eps_0>0$. There exist $\delta_0>0$ and $C>0$ such that when $0<\eps\leq\eps_0$ and $0<\Delta t\leq\delta_0$ if $(\bfx,\bfv)$ solves \eqref{eq:ODE} on $[0,T]$ with initial datum $(\bfx^0,\bfv^0)$ such that $\|\bfv^0\|\leq M$, 
and $(\bfx,e,\bfw)$ solves \eqref{scheme:CN} with initial datum $(\bfx^0,e^0,\bfw^0)$ such that $\bfw^0=\bfv^0$, $e^0=\tfrac12\|\bfv^0\|^2$, 
then for any $n$ such that $0\leq n\Delta t\leq T$,
\[
\left\|(\bfx^n,e^n)-\left(\bfx,\tfrac12\|\bfv\|^2\right)(n\,\Delta t)\right\|
\,+\,\eps\,\|\bfw^n-\bfv(n\,\Delta t)\|\\
\leq C\,\min\left(\left\{\,\eps+(\Delta t)^2\,;\,\dfrac{(\Delta t)^2}{\eps^5}\,\right\}\right)\,.
\]
\end{theorem}

To give a more concrete outcome we point out that
\[
\min\left(\left\{\,\eps+(\Delta t)^2\,;\,\dfrac{(\Delta t)^2}{\eps^5}\,\right\}\right)
\,\leq\,
\begin{cases}
\dfrac{(\Delta t)^2}{\eps^5}, &\textrm{if }\quad\Delta t\leq\eps^3,\quad\emph{i.e.}\quad(\Delta t)^{\frac13}\leq\eps\,,\\
2\,\eps, &\textrm{if }\quad\eps^3\leq\Delta t\leq\eps^{\frac12},\quad\emph{i.e.}\quad (\Delta t)^2\leq\eps\leq (\Delta t)^{\frac13}\,,\\
2\,(\Delta t)^2, &\textrm{if }\quad\eps^{\frac12}\leq\Delta t,\quad\emph{i.e.}\quad \eps\leq (\Delta t)^2\,,
\end{cases}
\]
whereas 
\[
\max_{0<\eps\leq \eps_0}\min\left(\left\{\,\eps+(\Delta t)^2\,;\,\dfrac{(\Delta t)^2}{\eps^5}\,\right\}\right)
\lesssim (\Delta t)^{\frac13}\,.
\]
Note that though the present analysis is more involved than the one of \cite{FiRo21} the outcome is similar to the one obtained there for second-order schemes and, in particular, matches the best expectations rationalized in \cite[Appendix~A]{FiRo21} for asymptotic preserving schemes of general oscillatory equations. Incidentally we stress that there are two main related ways in which our present analysis differ from the one in \cite{FiRo21}, both being responsible for new kinds of discretization consistency errors 
\begin{itemize}
\item the inhomogeneity of the magnetic field forces more chain rule applications in the normal form part of the analysis, this being immaterial at the continuous level but resulting in consistency errors that are due to failure of exact dicrete chain rules and which are denoted with $\tau_\sharp$ letters in discrete elimination lemmas, Lemmas~\ref{le:disc-1st}, \ref{le:disc-2nd}, \ref{le:disc-3rd} and~\ref{le:disc-1st-bis}
\item estimates on $\ue-\tfrac12\|\ubw\|^2$ (see Lemma~\ref{le:disc-cons}) should be thought as discretization consistency bounds and encode that the augmented formulation lets the scheme ensure the smallness of $\ue-\tfrac12\|\ubw\|^2$ when $\Delta t$ is smaller than $\eps^2$ and its concvergence to $\ug$ otherwise. 
\end{itemize} 

The $\eps+(\Delta t)^2$ part of the bound in Theorem~\ref{th:main} arises from 
\begin{itemize}
\item the $\eps$ proximity of $(\bfx,\tfrac12\|\bfv\|^2)(n\,\Delta t)$ and $(\bfy,g)(n\,\Delta t)$ (stated in Proposition~\ref{p:cont-x})
\item its discrete counterpart, the $\eps$ proximity of $(\bfx^n,e^n)$ and $(\bfy^n,g^n)$ (stated in Proposition~\ref{p:disc-x})
\item the $(\Delta t)^2$ proximity of $(\bfy,g)(n\,\Delta t)$ and $(\bfy^n,g^n)$ which results from the fact that \eqref{scheme:asymp} is a second-order scheme for \eqref{eq:asymp}.
\end{itemize}
In turn the $(\Delta t)^2\,\eps^{-5}$ part is a "direct" convergence analysis encoding how the stiffness impacts the second-order convergence of \eqref{scheme:CN} to \eqref{eq:aug} and the smallness of $\ue-\tfrac12\|\ubw\|^2$. Let us stress however that even the direct convergence analysis requires a fine understanding of the $\eps$ asymptotics so as to tame the stiffness impact in the stability part of the convergence analysis. A rough stability analysis would provide a bound exhibiting a dramatic exponential growth as a power of $1/\eps^2$.

In particular, even in the direct convergence part, our stability analysis involves the famous guiding-center variables $(\GC,\eGC)$,
$$
\left\{
  \begin{array}{l}
 \ds\GC :=\bfx-\,\eps\,\dfrac{\bfw^{\perp}}{b(\bfx)}\,,\\[0.8em]
\ds\eGC :=e\,+\,\dfrac{\eps}{b(\bfx)}\,\bE^\perp(\bfx)\cdot\bfw\,.
  \end{array}
\right.
$$
These variables are $\eps$ close to $(\bfx,e)$ but obey a slower dynamics. Indeed they remain $\eps^2$ close to a suitable solution of \eqref{eq:asymp}; see Proposition~\ref{p:cont-GC}. To keep our analysis as short as possible we have decided not include a convergence analysis of those. But we refer to \cite{FiRo21} for an example of the latter and we point out that our numerical tests also show a better capture of $(\GC,\eGC)$ following the rate
\[
\min\left(\left\{\,\eps^2+(\Delta t)^2\,;\,\dfrac{(\Delta t)^2}{\eps^4}\,\right\}\right)
\,\leq\,
\begin{cases}
\dfrac{(\Delta t)^2}{\eps^4}&\textrm{if }\quad\Delta t\leq\eps^3
\qquad\quad\emph{i.e.}\quad(\Delta t)^{\frac13}\leq\eps\,,\\
2\,\eps^2&\textrm{if }\quad\eps^3\leq\Delta t\leq\eps
\quad\emph{i.e.}\quad \Delta t\leq\eps\leq(\Delta t)^{\frac13}\,,\\
2\,(\Delta t)^2&\textrm{if }\quad\eps\leq\Delta t\,,
\end{cases}
\]
which is consistant with the general discussion in \cite[Appendix~A]{FiRo21}. Note that 
\[
\max_{0<\eps\leq \eps_0}\min\left(\left\{\,\eps^2+(\Delta t)^2\,;\,\dfrac{(\Delta t)^2}{\eps^4}\,\right\}\right)
\lesssim (\Delta t)^{\frac23}\,.
\]

\subsection*{Numerical tests}

We conclude this introduction with numerical experiments confirming the sharpness of our estimates. In our experiments we choose
\begin{align*}
b(\bfx)&=\,\dfrac{10}{\sqrt{100-\| \bfx \|^{2}}}\,,&
\phi(\bfx)&=\,\dfrac12\,\|\bfx\|^2\,.
\end{align*}
Note that neither $b$ nor $\phi$ meet the assumptions of our results since $b$ is unbounded and not lower bounded away from zero and $\phi$ is not bounded with bounded derivatives. Their treatments would require slightly more elaborate arguments. Yet the outcome would be similar, as confirmed by numerical tests.

In the present experiments we fix the final time $T=1$, the initial position to be $\bfx^{0} = (2,2)$ and the initial velocity to be $\bfv^{0} = (3,3)$. Moreover we constrain $\Delta t$ to be such that $N_{\Delta t}:=T/(\Delta t)$ is an integer and measure sizes of vectors $(a_n)_{1\leq n\leq N_{\Delta t}}$ through the discrete $\ell^1$ norm
\[
\vvvert (a_n)_{1\leq n\leq N_{\Delta t}}\vvvert_{\Delta t}:= 
\dfrac{1}{N_{\Delta t}}\sum_{n=1}^{N_{\Delta t}}\|a_n\|\,.
\]

\begin{figure}
        \centering
        {\includegraphics[width=0.49\linewidth]{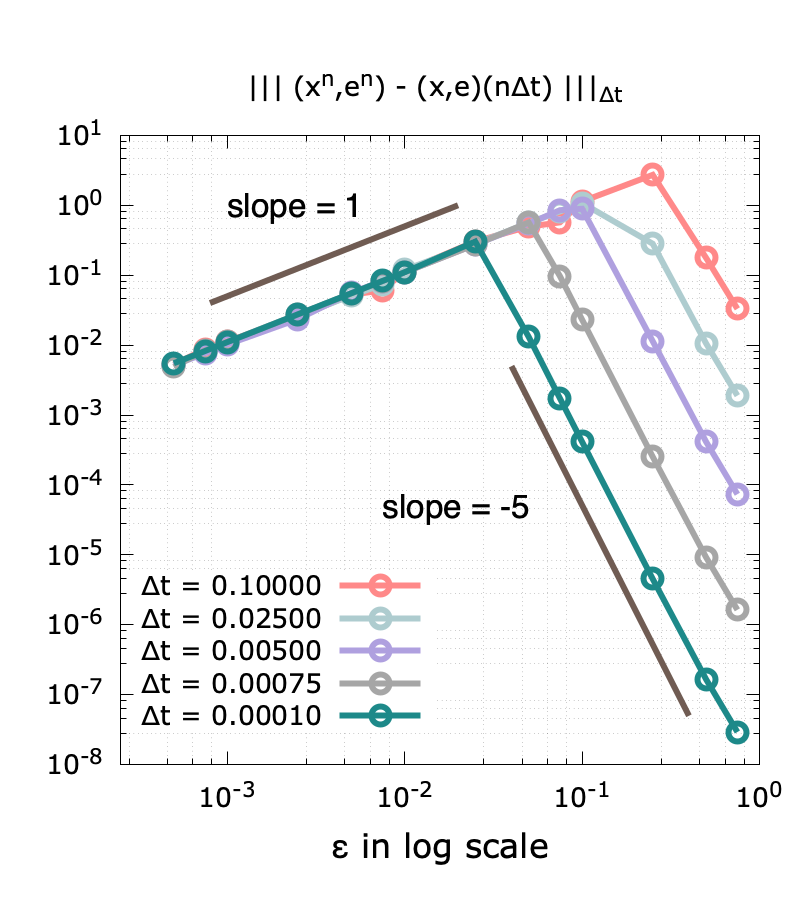}}
        {\includegraphics[width=0.49\linewidth]{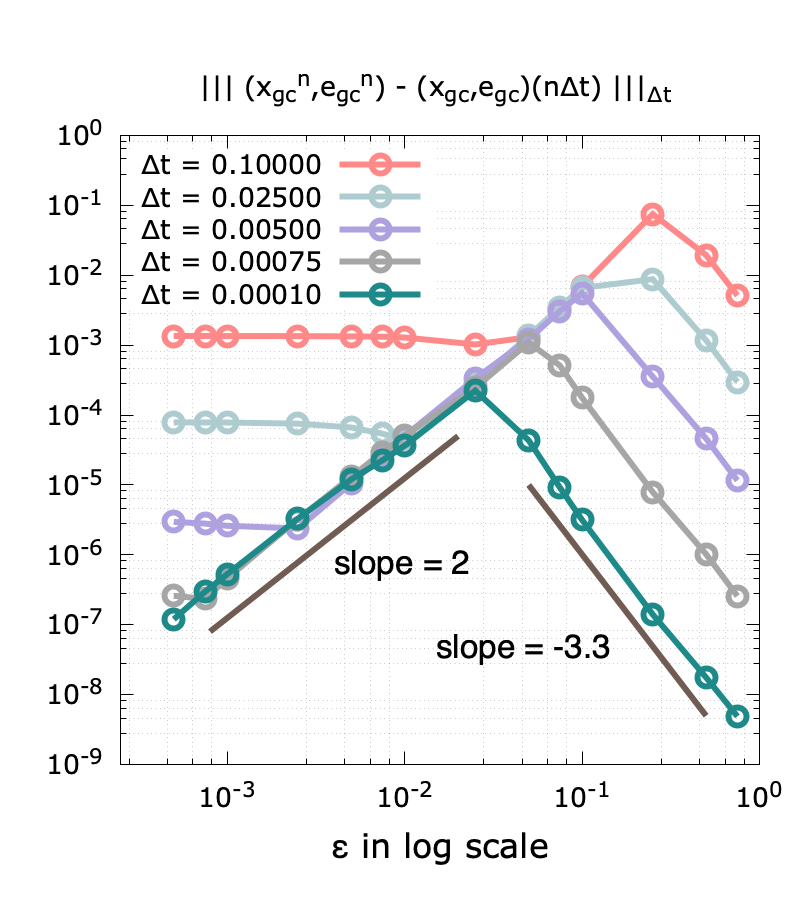}}
      	\caption{Convergence errors, that is, comparison of the solution to the scheme \eqref{scheme:CN} with the solution to \eqref{eq:ODE}. (Left) Comparison of the original slow variables $(\bfx,e)$. (Right) Comparison of their guiding center modifications $(\GC,\eGC)$.}
	\label{Fig:convergence}
\end{figure}

\begin{figure}
        \centering
        {\includegraphics[width=0.49\linewidth]{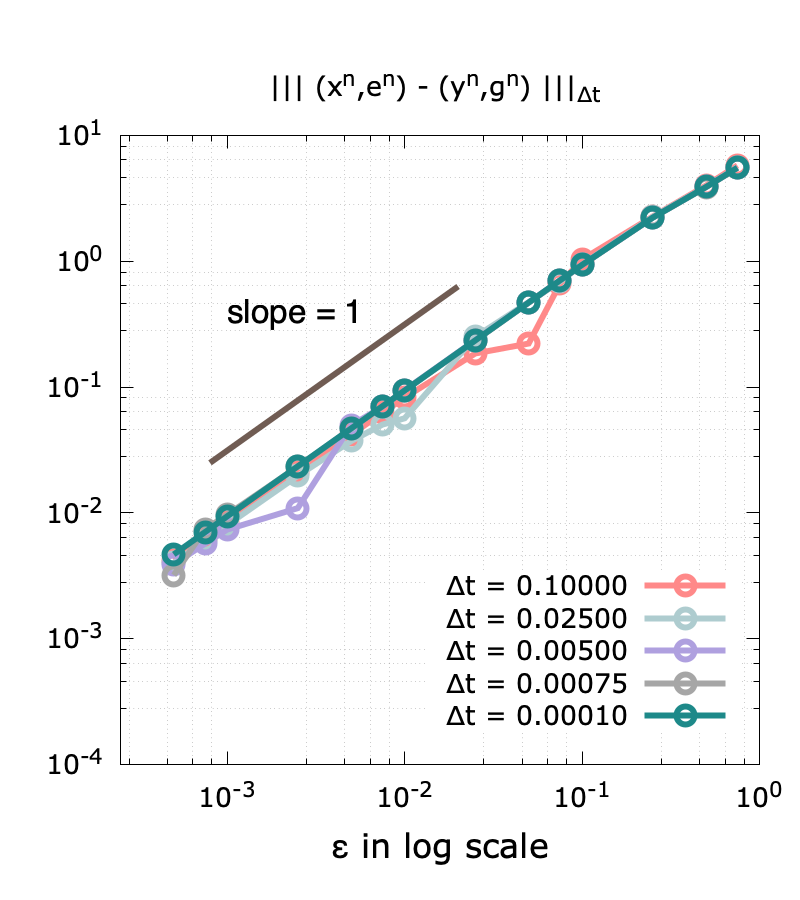}}
        {\includegraphics[width=0.49\linewidth]{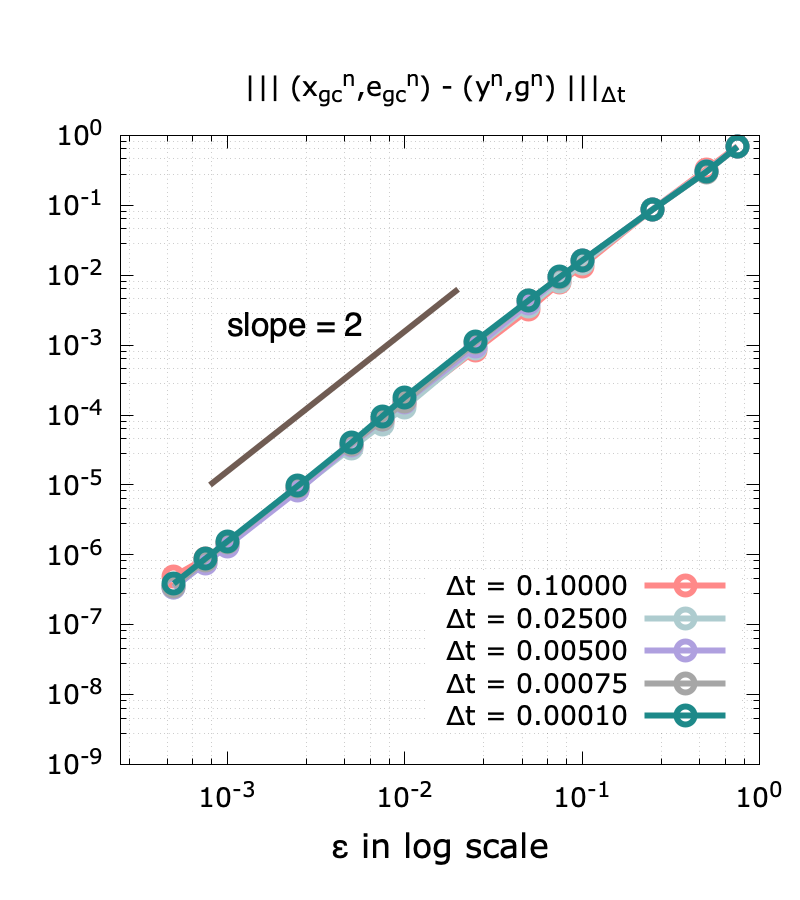}}
      	\caption{Discrete $\eps$ asymptotic errors, that is, comparison of the solution to the scheme \eqref{scheme:CN} with solutions to \eqref{scheme:asymp}. (Left) Comparison of the original slow variables $(\bfx,e)$. (Right) Comparison of their guiding center modifications $(\GC,\eGC)$. Note that initial data for \eqref{scheme:asymp} are set respectively to $(\bfx^0,e^0)$ for the computation of the left panel and to $(\GC^0,\eGC^0)$ for the computation of the right one.}
	\label{Fig:asymp-disc}
\end{figure}

\begin{figure}
        \centering
        {\includegraphics[width=0.49\linewidth]{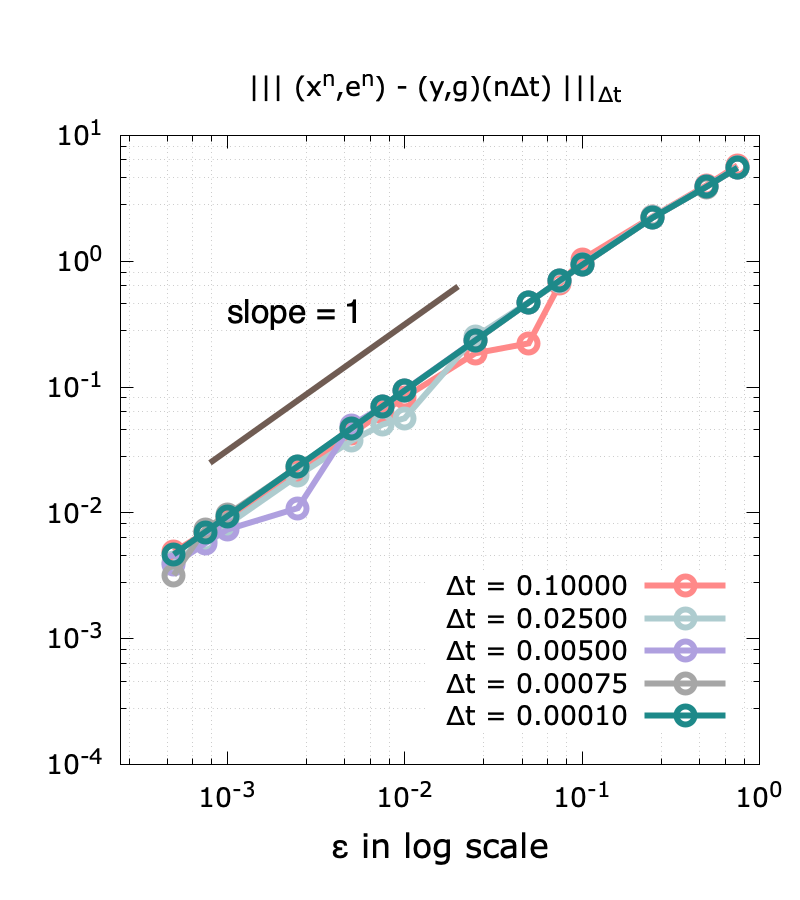}}
        {\includegraphics[width=0.49\linewidth]{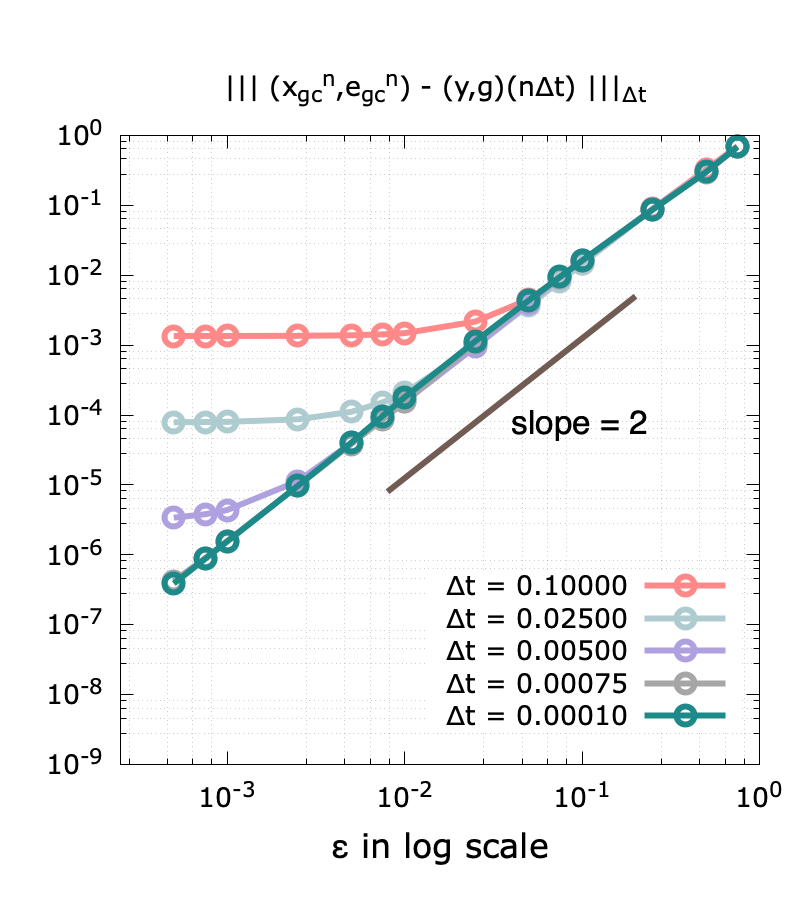}}
      	\caption{Continuous $\eps$ asymptotic errors, that is, comparison of the solution to the scheme \eqref{scheme:CN} with solutions to \eqref{eq:asymp}. (Left) Comparison of the original slow variables $(\bfx,e)$. (Right) Comparison of their guiding center modifications $(\GC,\eGC)$. Note that initial data for \eqref{scheme:asymp} are set respectively to $(\bfx^0,e^0)$ for the computation of the left panel and to $(\GC^0,\eGC^0)$ for the computation of the right one.}
	\label{Fig:asymp-cont}
\end{figure}

Numerical tests match quite closely theoretical bounds. For the sake of comparison, let us recall those here
\begin{itemize}
\item Figure~\ref{Fig:convergence}, (Left) $\min\left(\left\{\,\eps+(\Delta t)^2\,;\,(\Delta t)^2\,\eps^{-5}\,\right\}\right)$,\qquad
(Right) $\min\left(\left\{\,\eps^2+(\Delta t)^2\,;\,(\Delta t)^2\,\eps^{-4}\,\right\}\right)$,
\item Figure~\ref{Fig:asymp-disc}, (Left) $\eps$,\qquad (Right) $\eps^2$,
\item Figure~\ref{Fig:asymp-cont}, (Left) $\eps+(\Delta t)^2$,\qquad (Right) $\eps^2+(\Delta t)^2$\,.
\end{itemize}
There are two departures worth commenting. On one hand, some of the regimes of the bounds reachable only with very small $\eps$ are unseen. For instance this is the case of the $(\Delta t)^2$ regime for the bound on the convergence error in $(\bfx,e)$, which would be apparent if we were reaching the $\eps \lesssim (\Delta t)^2$ regime. This numerical limitation of our experiments is due to the fact that computing reference solutions with very small $\eps$ comes at a prohibitive price. On the other hand, for the convergence error in $(\GC,eGC)$, we measure $(\Delta t)^2\,\eps^{-3.3}$ in the regime where we should see $(\Delta t)^2\,\eps^{-4}$. This seems to be a purely numerical artifact.


\subsection*{Outline}

After the present introduction the rest of the paper is organized as follows. In the next section we prove $\eps$ asymptotics at the continuous level. After that, we begin the discrete analysis and provide at the discrete level, in this order, uniform bounds, $\eps$ asymptotics and discrete error bounds. Then we combine those to prove Theorem~\ref{th:main}. At last, in an appendix we prove that the implicit scheme may be solved without stiff constraint on the discretization.


\medskip

\noindent\emph{Acknowledgment.} KHT expresses his appreciation of the hospitality of IMT, Universit\'e Toulouse~III, during the preparation of the present contribution. 

\section{Continuous analysis}\label{s:continuous}

Our continuous analysis follows closely the strategy of \cite{FiRo20}: first prove uniform bounds then investigate the $\eps$ asymptotics by tracking how multilinear expressions in the fast variable may be split into a slow part and a smaller fast part. However, in contrast with the analyses of \cite{FiRo20}, the restriction made on the structure provides readily relevant uniform bounds at the continuous level. Indeed from the conservation of energy stems that solutions to \eqref{eq:ODE} satisfy for any $t$ and any $\eps>0$
\[
\|\bfv(t)\|
\leq \sqrt{ \|\bfv^0\|^2+4\,\|\phi\|_{L^\infty}}
\leq \|\bfv^0\|+2\,\sqrt{\|\phi\|_{L^\infty}}\,.
\]

\subsection{Original variables}\label{s:cont-x}

With the uniform bounds in hands we may turn to the elimination process, aiming at a normal form where slow and fast variables are sufficiently uncoupled. The keystone of the uncoupling is
\begin{align} \label{eq:velocity}
\bfv\, = \, \eps^{2} \, \dfrac{\dD}{{\dD}t} \left( \dfrac{\bfv^{\perp}}{b(\bfx)} \right) \,-\, \eps \,\dD_{\bfx}\left(\dfrac1b\right)(\bfx)(\bfv)\,\bfv^{\perp} - \eps \dfrac{\bE^{\perp}}{b}(\bfx)\,,
\end{align}
which is readily derived from \eqref{eq:ODE}. To begin with, inserting \eqref{eq:velocity} back into the first equation of \eqref{eq:ODE} yields
\begin{align}\label{eq:cont-x1}
\dfrac{\dD}{{\dD}t} \left( \bfx- \eps\, \dfrac{\bfv^{\perp}}{b(\bfx)} \right) \,=\,  \,-\,\dD_{\bfx}\left(\dfrac1b\right)(\bfx)(\bfv)\,\bfv^{\perp} -\dfrac{\bE^{\perp}}{b}(\bfx)\,.
\end{align}

At this stage, in order to proceed with the uncoupling, we need to eliminate quadratic terms. Before doing so it is instructive to slightly generalize the elimination of $\bfv$ encoded by \eqref{eq:velocity} as elimination of linear terms. Combining \eqref{eq:velocity} with direct differentiation proves the following lemma.

\begin{lemma} \label{le:cont-1st}
Consider $\bL$ a differentiable function of time with values in linear maps from $\R^2$ to some $\R^n$, $n\in\N^*$. Then solutions to \eqref{eq:ODE} satisfy
\begin{equation*}
\bL(t)(\bfv(t)) \, = \,-\eps^{2}\dfrac{\dD \kappa_{\bL}(\cdot,\bfx,\bfv)}{{\dD}t}(t) \,+\,\eps\,\eta_{\bL}(t,\bfx(t),\bfv(t)) \,,
\end{equation*}
where
\begin{align*} 
\kappa_{\bL}(t, \bfx, \bfv) & \,:=\,-\dfrac{1}{b(\bfx)}\bL(t)(\bfv^{\perp})\,,\\
\end{align*}
and
\begin{align*} 
  \eta_{\bL}(t, \bfx, \bfv) &\,:=\, -\,\dD_{\bfx}\left(\dfrac1b\right)(\bfx)(\bfv)\,\bL(t)(\bfv^{\perp})
-\dfrac{1}{b(\bfx)}\,\bL(t)(\bE^{\perp}(\bfx))
-\frac{\eps}{b(\bfx)}\,\dfrac{\dD \bL}{{\dD}t}(t)(\bfv^{\perp})\,.
\end{align*}
\end{lemma}

Above, to derive \eqref{eq:cont-x1}, we have just used the case when $\bL$ is constant equal to the identity map.

Quadratic terms are not eliminated in the same way as linear terms are. They leave slow contributions. The starting observation in this direction is that if $\bB$ is a bilinear form on $\R^2$ then for any vector $\bfw$
\begin{equation}\label{eq:trace}
\bB(\bfw,\bfw)+\bB(\bfw^\perp,\bfw^\perp)\,=\,\|\bfw\|^2\,\Tr(\bB)\,,
\end{equation}
where $\Tr$ denotes the trace operator. Note also that for solutions to \eqref{eq:ODE}, $e:=\frac12\|\bfv\|^2$ is exactly as slow as $\bfx$ is, as may deduced from conservation of energy.

The following lemma provides a quadratic counterpart to Lemma~\ref{le:cont-1st}

\begin{lemma} \label{le:cont-2nd}
Consider $\bB$ a differentiable function of time with values in bilinear maps from $\R^2\times\R^2$ to some $\R^n$, $n\in\N^*$. Then solutions to \eqref{eq:ODE} satisfy
\begin{equation*}
\bB(t)(\bfv(t),\bfv(t)) \, = \,
e(t)\,\Tr(\bB(t))
\,-\,\eps^{2}\dfrac{\dD \kappa_{\bB}(\cdot,\bfx,\bfv)}{{\dD}t}(t) \,+\,\eps\,\eta_{\bB}(t,\bfx(t),\bfv(t)) \,,
\end{equation*}
where
\begin{align*} 
\kappa_{\bB}(t, \bfx, \bfv) & \,:=\,-\frac12\,\dfrac{1}{b(\bfx)}\bB(t)(\bfv,\bfv^{\perp})\,,\\
\end{align*}
and
\begin{align*} 
  \eta_{\bB}(t, \bfx, \bfv) &\,:=\, -\,\frac12\,\dD_{\bfx}\left(\dfrac1b\right)(\bfx)(\bfv)\,\bB(t)(\bfv,\bfv^{\perp})
-\frac12\frac{\eps}{b(\bfx)}\,\dfrac{\dD \bB}{{\dD}t}(t)(\bfv,\bfv^{\perp})\\
&\qquad-\frac12\dfrac{1}{b(\bfx)}\,\left(\bB(t)(\bfv,\bE^{\perp}(\bfx))+\bB(t)(\bE(\bfx),\bfv^{\perp})\right)\,.
\end{align*}
\end{lemma}

\begin{proof}
A direct computation relying again on \eqref{eq:velocity} shows 
\begin{equation*}
\frac12\bB(t)(\bfv(t),\bfv(t))-\frac12\bB(t)(\bfv(t)^\perp,\bfv(t)^\perp)\, = \,
-\eps^{2}\dfrac{\dD \kappa_{\bB}(\cdot,\bfx,\bfv)}{{\dD}t}(t) \,+\,\eps\,\eta_{\bB}(t,\bfx(t),\bfv(t)) \,,
\end{equation*}
whereas \eqref{eq:trace} implies
\[
\bB(t)(\bfv(t),\bfv(t))\,=\,
e(t)\,\Tr(\bB(t))+\frac12\bB(t)(\bfv(t),\bfv(t))-\frac12\bB(t)(\bfv(t)^\perp,\bfv(t)^\perp)\,.
\]
\end{proof}

Coming back to \eqref{eq:cont-x1} we apply Lemma~\ref{le:cont-2nd} with $\bB_1(t):=\cB_1(\bfx(t))$
\[
\cB_1(\bfx)(\bfw_1,\bfw_2):=
-\dD_{\bfx}\left(\dfrac1b\right)(\bfx)(\bfw_1)\,\bfw_2^{\perp}
\]
and observe that for this bilinear form setting $\bfw=\nabla_{\bfx}(b^{-1})(\bfx(t))$ in \eqref{eq:trace} yields
\[
\Tr(\bB_1(t))\,=\,-\nabla_{\bfx}^\perp\left(\dfrac1b\right)(\bfx(t))\,.
\]
Therefore solutions to \eqref{eq:ODE} satisfy
\begin{align}\label{eq:cont-x2}
\dfrac{\dD}{{\dD}t} \left( \bfx\,+\,\eps\,\kappa_1(\bfx,\bfv)
\right) \,=\,-e\,\nabla_{\bfx}^\perp\left(\dfrac1b\right)(\bfx)-\dfrac{\bE^{\perp}}{b}(\bfx)
+\eps\,\eta_1(\bfx,\bfv)\,,
\end{align}
with
\begin{align*} 
\kappa_1(\bfx, \bfv) & \,:=\,-\,\dfrac{\bfv^{\perp}}{b(\bfx)}-\,\frac{\eps}{2}\,\dfrac{1}{b(\bfx)}\cB_1(\bfx)(\bfv,\bfv^{\perp})\,,\\
\end{align*}
and
\begin{align*}
  \eta_1(\bfx, \bfv) &\,:=\,
\frac14\,\dD_{\bfx}^2\left(\dfrac{1}{b^2}\right)(\bfx)(\bfv,\bfv)\,\bfv
-\frac12\dfrac{1}{b(\bfx)}\,\left(\cB_1(\bfx)(\bfv,\bE^{\perp}(\bfx))+\cB_1(\bfx)(\bE(\bfx),\bfv^{\perp})\right)\,.
\end{align*}

We have reached the stage when we may derive the $\eps$ asymptotic behavior of $(\bfx,e)$.

\begin{proposition}\label{p:cont-x}
Let $T>0$, $M>0$ and $\eps_0>0$. There exists $C>0$ such that when $0<\eps\leq\eps_0$ if $(\bfx,\bfv)$ solves \eqref{eq:ODE} with initial datum $(\bfx^0,\bfv^0)$ such that $\|\bfv^0\|\leq M$ and $(\bfy,g)$ solves \eqref{eq:asymp} with initial datum $(\bfy^0,g^0)=(\bfx^0,\tfrac12\|\bfv^0\|^2)$ then with $e:=\tfrac12\|\bfv\|^2$ for any $0\leq t\leq T$,
\[
\|(\bfx(t),e(t))-(\bfy(t),g(t))\|\leq C\,\eps\,.
\]
\end{proposition}

The dependence of $C$ on $T$ and $\|\bfv^0\|$ could actually be made rather explicit ; see the related analyses in \cite{FiRo20,FiRo21} but we have decided to cut technicalities to a bare minimum. 

\begin{proof}
Note that from the conservation of energy one deduces that for some $M'$ depending only on $M$ (and $\phi$), $e$, $\bfv$ and $g$ remain bounded by $M'$ when $\|\bfv^0\|\leq M$. Then observe that 
\begin{align*}
(\bfz,\bfw)&\mapsto \kappa_1(\bfz,\bfw)\,,&
(\bfz,\bfw)&\mapsto \eta_1(\bfz,\bfw)
\end{align*}
are uniformly bounded and 
\[
(\bfz,h)\mapsto -h\,\nabla_{\bfx}^\perp\left(\dfrac1b\right)(\bfz)-\dfrac{\bE^{\perp}}{b}(\bfz)
\]
is uniformly Lipschitz when restricted respectively to $\|\bfw\|\leq M'$ and $|h|\leq M'$. 

By integrating \eqref{eq:cont-x2} one deduces that for some $K$ depending only on $M>0$ and $\eps_0>0$ (and the fields), for any $0\leq s\leq t$,
\[
\left\| \bfx(t)-\bfx(s)+\int_s^t \left(e(\tau)\,\nabla_{\bfx}^\perp\left(\dfrac1b\right)(\bfx(\tau))+\dfrac{\bE^{\perp}}{b}(\bfx(\tau))\right)\,\dD \tau \right\|
\,\leq\,K\,\eps\,(1+t-s)\,.
\]
In turn note that for any $0\leq s\leq t$,
\[
\bfy(t)-\bfy(s)+\int_s^t \left(g(\tau)\,\nabla_{\bfx}^\perp\left(\dfrac1b\right)(\bfy(\tau))+\dfrac{\bE^{\perp}}{b}(\bfy(\tau))\right)\,\dD \tau
\]
and for any $t$, $e(t)-g(t)=-(\phi(\bfx(t))-\phi(\bfy(t)))$. Substracting and using the aforementioned Lipschitz bound yield that for some $K'$ depending only on $M>0$ and $\eps_0>0$ (and the fields), for any $t\geq0$, 
\[
\|\bfy(t)-\bfx(t)\|\,\leq\,K'\int_0^t \|\bfy(s)-\bfx(s)\|\,\dD s\,+\,K\,\eps\,(1+t)\,.
\]
The proof is then concluded by the Gr\"onwall lemma with $C=\eD^{K'\,T}\,K\,(1+T)\,(1+\|\nabla_{\bfx}\phi\|_{L^\infty})$ since the bound on $e-g$ stems from the bound on $\bfx-\bfy$ and the energy conservation.
\end{proof}

\subsection{Guiding-center variables}\label{s:cont-gc}

The above computations also suggest that the evolution of the guiding center variable
\[
\GC:=\bfx-\,\eps\,\dfrac{\bfv^{\perp}}{b(\bfx)}
\]
could be described as part of a slow uncoupled dynamics up to $\cO(\eps^2)$ errors.

In order to achieve this we need to refine \eqref{eq:cont-x2} and the energy conservation so as to get rid of $\cO(\eps)$ terms involving $\bfv$ either appearing explicitly in $\eps\,\eta_1$ or arising from the replacement of $\bfx$ with $\GC$ in $\cO(1)$ terms. Heeding the shape of $\eta_1$  we need a trilinear version of Lemmas~\ref{le:cont-1st} and~\ref{le:cont-2nd}.

\begin{lemma} \label{le:cont-3rd}
Consider $\bT$ a differentiable function of time with values in \emph{symmetric} trilinear maps from $\R^2\times\R^2\times\R^2$ to some $\R^n$, $n\in\N^*$. Then solutions to \eqref{eq:ODE} satisfy
\begin{equation*}
\bT(t)(\bfv(t),\bfv(t),\bfv(t)) \, = \,
\,-\,\eps^{2}\dfrac{\dD \kappa_{\bT}(\cdot,\bfx,\bfv)}{{\dD}t}(t) \,+\,\eps\,\eta_{\bT}(t,\bfx(t),\bfv(t)) \,,
\end{equation*}
where
\begin{align*} 
\kappa_{\bT}(t, \bfx, \bfv) & \,:=\,-\,\dfrac{1}{b(\bfx)}\left(\bT(t)(\bfv,\bfv,\bfv^{\perp})+\frac23\,\bT(t)(\bfv^\perp,\bfv^\perp,\bfv^{\perp})\right)\,,\\
\end{align*}
and
\begin{align*} 
  \eta_{\bT}(t, \bfx, \bfv) &\,:=\, -\,\dD_{\bfx}\left(\dfrac1b\right)(\bfx)(\bfv)\,
\left(\bT(t)(\bfv,\bfv,\bfv^{\perp})+\frac23\,\bT(t)(\bfv^\perp,\bfv^\perp,\bfv^{\perp})\right)\\
&\qquad
-\frac{\eps}{b(\bfx)}\,
\left(\dfrac{\dD \bT}{{\dD}t}(t)(\bfv,\bfv,\bfv^{\perp})+\frac23\,\dfrac{\dD \bT}{{\dD}t}(t)(\bfv^\perp,\bfv^\perp,\bfv^{\perp})\right)\\
&\qquad-\dfrac{1}{b(\bfx)}\,\left(\,2\,\bT(t)(\bE(\bfx),\bfv,\bfv^\perp)+\bT(t)(\bE^{\perp}(\bfx),\bfv,\bfv)+2\,\bT(t)(\bE^\perp(\bfx),\bfv^{\perp},\bfv^\perp)\,\right)\,.
\end{align*}
\end{lemma}

Let us set 
\[
\eGC:=e\,+\,\dfrac{\eps}{b(\bfx)}\,\dD_{\bfx}\phi(\bfx)\left(\bfv^{\perp}\right)
\,=\,e\,+\,\dfrac{\eps}{b(\bfx)}\,\bE^\perp(\bfx)\cdot\bfv
\]
so that solutions to \eqref{eq:ODE} satisfy
\begin{equation}\label{eq:cont-e}
\dfrac{\dD}{\dD t}\left(\eGC+\phi(\GC)+\eps^2\,\kappa_e(\bfx,\bfv)\right)\,=\,0
\end{equation}
with
\begin{align*}
\kappa_e(\bfx,\bfv)
\,:=\,-\,\dfrac{1}{b(\bfx)^2}\int_0^1 \dD^2_{\bfx} \phi\left(\bfx-\eps\,\tau\,\dfrac{\bfv^\perp}{b(\bfx)}\right)
\left(\bfv^\perp,\bfv^\perp\right)
\,(1-\tau)\,\dD \tau\,.
\end{align*}

Now we apply Lemma~\ref{le:cont-1st} with $\bL_2(t):=\cL_2(\bfx(t),\GC(t),e(t))$ where
\begin{align*}
\cL_2(\bfx,\bfy,e)(\bfw)
&\,:=\,-\dfrac{e}{b(\bfx)}\,\dD\left(\nabla_{\bfx}^\perp\left(\dfrac1b\right)\right)(\bfx)(\bfw^\perp)
+\dfrac{\bE^\perp}{b}(\bfx)\cdot\bfw\quad\nabla_{\bfx}^\perp\left(\dfrac1b\right)(\bfy)\\
&\quad
-\dfrac{1}{b(\bfx)}\dD_\bfx\left(\dfrac{\bE^\perp}{b}\right)(\bfx)(\bfw^\perp)
-\frac12\dfrac{1}{b(\bfx)}\,\left(\cB_1(\bfx)(\bfw,\bE^{\perp}(\bfx))+\cB_1(\bfx)(\bE(\bfx),\bfw^{\perp})\right)
\end{align*}
and Lemma~\ref{le:cont-3rd} with $\bT_1(t):=\cT_1(\bfx(t))$ where
\begin{align*}
\cT_1&(\bfx)(\bfw_1,\bfw_2,\bfw_3)\\
&\,:=\,
\frac{1}{12}\,
\left(\dD_{\bfx}^2\left(\dfrac{1}{b^2}\right)(\bfx)(\bfw_1,\bfw_2)\,\bfw_3
+\dD_{\bfx}^2\left(\dfrac{1}{b^2}\right)(\bfx)(\bfw_3,\bfw_1)\,\bfw_2
+\dD_{\bfx}^2\left(\dfrac{1}{b^2}\right)(\bfx)(\bfw_2,\bfw_3)\,\bfw_1
\right)\,.
\end{align*}
This implies that solutions to \eqref{eq:ODE} satisfy
\begin{align}\label{eq:cont-gc}
\dfrac{\dD}{{\dD}t} \left( \GC\,+\,\eps^2\,\kappa_2(\bfx,\GC,e,\bfv)
\right) \,=\,-\eGC\,\nabla_{\bfx}^\perp\left(\dfrac1b\right)(\GC)-\dfrac{\bE^{\perp}}{b}(\GC)
+\eps^2\,\eta_2(\bfx,\GC,e,\bfv)\,,
\end{align}
with $\kappa_2$ given by
\begin{align*} 
\kappa_2(\bfx,\GC,e,\bfv) & \,:=\,-\,\frac{1}{2}\,\dfrac{1}{b(\bfx)}\cB_1(\bfx)(\bfv,\bfv^{\perp})\\
&\quad\,-\,\dfrac{\eps}{b(\bfx)}
\left(\cL_2(\bfx,\GC,e)(\bfv^{\perp})
+\cT_1(\bfx)(\bfv,\bfv,\bfv^{\perp})+\tfrac23\,\cT_1(\bfx)(\bfv^\perp,\bfv^\perp,\bfv^{\perp})\right)\,,\\
\end{align*}
whereas $\eta_2$ is
\begin{align*}
  \eta_2(\bfx,\GC,e,\bfv) &\,:=\,
\dfrac{1}{b(\bfx)^2}\int_0^1 \dD^2_{\bfx}\left(e\,\nabla_{\bfx}^\perp\left(\dfrac1b\right)
+\dfrac{\bE^\perp}{b}\right)\left(\bfx-\eps\,\tau\,\dfrac{\bfv^\perp}{b(\bfx)}\right)
\left(\bfv^\perp,\bfv^\perp\right)
\,(1-\tau)\,\dD \tau\\
&\quad
-\,\left(\dD_{\bfx}\left(\dfrac{\cL_2}{b}\right)(\bfx,\GC,e)(\bfv)\right)\,(\bfv^{\perp})
-\dfrac{\cL_2(\bfx,\GC,e)}{b(\bfx)}\,(\bE^{\perp}(\bfx))\\
&\quad-\frac{1}{b(\bfx)}\,\left(\dD_e \cL_2(\bfx,\GC,e)(\bE(\bfx)\cdot\bfv)\right)(\bfv^{\perp})\\
&\quad+\frac{\eps}{b(\bfx)}\,\left(\dD_{\bfy} \cL_2(\bfx,\GC,e)\left(
\dD_{\bfx}\left(\dfrac1b\right)(\bfx)(\bfv)\,\bfv^{\perp}
+\dfrac{\bE^{\perp}}{b}(\bfx)\right)\right)(\bfv^{\perp})\\
&\quad-\,
\left(\left(\dD_{\bfx}\left(\dfrac{\cT_1}{b}\right)(\bfx)(\bfv)\right)(\bfv,\bfv,\bfv^{\perp})+\frac23\,
\left(\dD_{\bfx}\left(\dfrac{\cT_1}{b}\right)(\bfx)(\bfv)\right)(\bfv^\perp,\bfv^\perp,\bfv^{\perp})\right)\\
&\quad-\,\left(\,2\,\dfrac{\cT_1}{b}(\bfx)(\bE(\bfx),\bfv,\bfv^\perp)+\dfrac{\cT_1}{b}(\bfx)(\bE^{\perp}(\bfx),\bfv,\bfv)+2\,\dfrac{\cT_1}{b}(\bfx)(\bE^\perp(\bfx),\bfv^{\perp},\bfv^\perp)\,\right)\,.
\end{align*}

From here a slight variation on the proof of Proposition~\ref{p:cont-x} proves the following.

\begin{proposition}\label{p:cont-GC}
Let $T>0$, $M>0$ and $\eps_0>0$. There exists $C>0$ such that when $0<\eps\leq\eps_0$ if $(\bfx,\bfv)$ solves \eqref{eq:ODE} with initial datum $(\bfx^0,\bfv^0)$ such that $\|\bfv^0\|\leq M$ and $(\bfy,g)$ solves \eqref{eq:asymp} with initial datum 
\[
(\bfy^0,g^0)=\left(\bfx^0-\,\dfrac{\eps}{b(\bfx^0)}\,(\bfv^0)^\perp,\quad
\frac12\|\bfv^0\|^2
+\,\dfrac{\eps}{b(\bfx^0)}\,\bE^\perp(\bfx^0)\cdot\bfv^0\right)
\]
then with
\begin{align*}
\GC&:=\bfx-\,\dfrac{\eps}{b(\bfx)}\,\bfv^\perp\,,&
\eGC&:=\,\frac12\|\bfv\|^2+\,\dfrac{\eps}{b(\bfx)}\,\bE^\perp(\bfx)\cdot\bfv\,,&
\end{align*}
for any $0\leq t\leq T$,
\[
\|(\GC(t),\eGC(t))-(\bfy(t),g(t))\|\leq C\,\eps^2\,.
\]
\end{proposition}

\section{Discrete uniform bounds}\label{s:disc-unif}

From now on we shall focus on discrete analysis. It is useful to introduce a few more pieces of notation so as to gain in writing concision. We shall denote composition with a subscript
\[
\bG_{\bfz}^\ell:=\bG(\bfz^\ell)
\]
so that for instance
\begin{align*}
\bG_{\ubz}^{n+1/2}&=\bG\left(\dfrac{\bfz^{n+1}+\bfz^n}{2}\right)\,,&
\underline{\bG_{\bfz}}^{n+1/2}&=\dfrac{\bG(\bfz^{n+1})+\bG(\bfz^n)}{2}\,.&
\end{align*}
Moreover we introduce the matrix $\bJ$ associated with ${}^\perp$
\begin{align} \label{notation_J}
\bJ&=\begin{pmatrix}
        0 &-1 \\
        1 &0
    \end{pmatrix}\,,&
\end{align}    
so that $\bfz^{\perp} \,=\, \bJ \bfz$. $\bJ$ is skew-adjoint and unitary. In particular $\bJ^{-1}=-\bJ$. For concision's sake it is also useful to introduce
\[
\bF:=-b^{-1}\bJ\bE\,.
\]

The following lemma gathers key observations on $\bJ$ towards uniform bounds and stability.

\begin{lemma} \label{le:stab}
\begin{enumerate}
\item For any real $\alpha$, $\Id+\alpha\bJ$ is invertible and for any vector $\bfz$
\begin{align*}
\|(\Id+\alpha\bJ)\,\bfz\|&=\sqrt{1+\alpha^2}\,\|\bfz\|\,,&
\|(\Id+\alpha\bJ)^{-1}\,\bfz\|&=\dfrac{1}{\sqrt{1+\alpha^2}}\,\|\bfz\|\,.
\end{align*}
\item For any real $\alpha$, $(\Id+\alpha\bJ)^{-1}(\Id-\alpha\bJ)$ is unitary.
\item For any real $\alpha$, $\alpha')$,
\[
(\Id+\alpha\bJ)^{-1}
-(\Id+\alpha'\bJ)^{-1}
\,=\,-(\Id+\alpha\bJ)^{-1}\,(\alpha-\alpha')\bJ\,(\Id+\alpha'\bJ)^{-1}\,.
\]
\end{enumerate}
\end{lemma}

\begin{proof}
The first norm equality stems from the Pythagoras' equality and the fact that $\bJ$ is skew-adjoint and unitary. It implies that $\Id+\alpha\bJ$ is one-to-one hence invertible, the second norm equality and the unitary property. 

The comparison equality may be derived by composing from the left with $(\Id+\alpha\bJ)^{-1}$ and from the right with $(\Id+\alpha'\bJ)^{-1}$ the simple $(\Id+\alpha'\bJ)-(\Id+\alpha\bJ)\,=\,-\,(\alpha-\alpha')\bJ$.
\end{proof}

The goal of the present section is to prove the following proposition.

\begin{proposition}\label{p:unif}
Let $T>0$, $M>0$, $\eps_0>0$. There exist $\delta_0>0$, $M'>0$ and $C>0$ such that when $0<\eps\leq\eps_0$ and $0<\Delta t\leq \delta_0$ if $(\bfx,e,\bfw)$ solves \eqref{scheme:CN} with initial datum $(\bfx^0,e^0,\bfw^0)$ such that $\|\bfw^0\|\leq M$, $|e^0|\leq \tfrac12M^2$ then when $0\leq n\,\Delta t\leq T$,
$$
|e^n|  \quad+\quad \|\bfw^n\|  \,\leq\, M'\,,
$$
and
$$
\left\{
\begin{array}{l}
\ds\|\ubw^{n+1/2}\|\,\leq\,
                  C\,\left(\frac{1}{\sqrt{1+\lambda^2}}+\eps\,\right)\,,
  \\[0.8em] 
\ds\|\bfx^{n+1}-\bfx^n\|
+\eps\|\bfw^{n+1}-\bfw^n\| \,\leq\, C\,\min\left(\left\{\dfrac{\Delta t}{\eps};\Delta t+\eps\right\}\right)\,,
\end{array}\right.
$$
where $\lambda:=\Delta t/\eps^2$.
\end{proposition}

A closer inspection of the proof, which we omit for the sake of readability, reveals that $\delta_0$ depends on the fields but not on $(T,M,\eps_0)$.

Let us begin the proof by observing that as in the continuous case the bound on $e^n$ follows readily from the energy conservation provided that $M'\geq \tfrac12M^2+2\|\phi\|_\infty$. The proofs of the rest of the bounds require more work and are scattered in the following subsections.

\subsection*{Bound on $\ubw^{n+1/2}$} 

To begin with we set for any $n\geq0$
\[
\bfz^n:=\ubw^{n+1/2}\,-\,\eps\,\bF_{\ubx}^{n+1/2}
\]
so that 
\[
\bfw^{n+1}-\bfw^n\,=\,-\lambda\,b_{\ubx}^{n+1/2}\,\bJ\,\bfz^n\,.
\]
Since for any $n\geq0$, $\bfw^{n+2}-\bfw^n=2\,(\ubw^{n+3/2}-\ubw^{n+1/2})$, we deduce by substracting two consecutive steps that for any $n\geq0$
\[
\bfz^{n+1}-\bfz^n
\,=\,-\eps\,(\bF_{\ubx}^{n+3/2}-\bF_{\ubx}^{n+1/2})
-\tfrac12\lambda\,b_{\ubx}^{n+3/2}\,\bJ\,\bfz^{n+1}
+\tfrac12\lambda\,b_{\ubx}^{n+1/2}\,\bJ\,\bfz^n\,,
\]
or in other words
\[
\left(\Id+\tfrac12\lambda\,b_{\ubx}^{n+3/2}\,\bJ\right)\bfz^{n+1}
\,=\,
\left(\Id+\tfrac12\lambda\,b_{\ubx}^{n+1/2}\,\bJ\right)\bfz^{n}
-\eps\,(\bF_{\ubx}^{n+3/2}-\bF_{\ubx}^{n+1/2})\,.
\]

At this stage, Lemma~\ref{le:stab} motivates the introduction of 
\[
\zeta^n:=\sqrt{1+\tfrac14\lambda^2\,(b_{\ubx}^{n+1/2})^2}\,\|\bfz^{n}\|
\]
so that for any $n\geq0$
\[
\zeta^{n+1}
\leq \zeta^n+\eps\,\|\bF_{\ubx}^{n+3/2}-\bF_{\ubx}^{n+1/2}\|\,.
\]
To proceed we observe that
\begin{align*}
\eps\,\|\bF_{\ubx}^{n+3/2}-\bF_{\ubx}^{n+1/2}\|
&\lesssim \eps\,\|\ubx^{n+3/2}-\ubx^{n+1/2}\|
\lesssim \eps\,\|\bfx^{n+2}-\bfx^{n+1}\|+\eps\,\|\bfx^{n+1}-\bfx^{n}\|\\
&\lesssim \Delta t\,\left(\|\bfz^{n+1}\|+\|\bfz^{n}\|+\eps\,+\eps\,\|\ubw^{n+3/2}\|^2
+\eps\,\|\ubw^{n+1/2}\|^2\right)\\
&\lesssim \Delta t\,\left(\zeta^{n+1}+\zeta^{n}
+\eps+\eps\,\left(\|\bfw^{n+2}\|^2+\|\bfw^{n+1}\|^2+\|\bfw^{n}\|^2\right)\right)\,.
\end{align*}

To complete the analysis we study now the variations of $\eps\,\|\bfw^{n}\|^2$. To do so we first observe that by taking the scalar product of the $\bfw$ equation of \eqref{scheme:CN} with $\ubw^{n+1/2}$ one obtains
\[
\eps\,\|\bfw^{n+1}\|^2-\eps\,\|\bfw^{n}\|^2
\,=\,2\,\Delta t\,\bE_{\bfx}^{n+1/2}\cdot \ubw^{n+1/2}\,,
\]
so that
\[
\eps\,\|\bfw^{n+1}\|^2-\eps\,\|\bfw^{n}\|^2
\lesssim \Delta t\,\left(\zeta^{n}+\eps\right)
\]

Introducing for any $n$
\[
Z^n:=\zeta^n+\eps+\eps\,\|\bfw^{n+1}\|^2+\eps\,\|\bfw^{n}\|^2
\]
and gathering the above estimates we deduce that for some constant $K$ and any $n$
\[
Z^{n+1}
\leq Z^n
+K\,\Delta t\,\left(
Z^{n+1}+Z^{n}
\right)\,.
\]
Assuming that $K\Delta t\leq 1/2$, one derives that for any $n$
\[
Z^{n+1}
\leq \frac{1+K\Delta t}{1-K\Delta t}\,Z^n = \left(1+\frac{2K\Delta t}{1-K\Delta t}\right)\,Z^n
\leq (1+4K\Delta t)\,Z^n\,.
\]
By iteration one deduces that for any $n$
\[
\|\ubw^{n+1/2}\|
\lesssim \|\bfz^n\|+\eps
\lesssim \frac{Z^n}{\sqrt{1+\lambda^2}}+\eps
\lesssim \frac{(1+4K\Delta t)^n\,Z^0}{\sqrt{1+\lambda^2}}+\eps
\lesssim \eD^{4K\,n\,\Delta t}\left(\|\ubw^{1/2}\|+\eps\right)\,.
\]

To conclude the bound on $\|\ubw^{n+1/2}\|$, there only remains to estimate $\|\ubw^{1/2}\|$. Once again we go back to the $\bfw$ equation of \eqref{scheme:CN} and deduce this time
\[
\ubw^{1/2}\,=\,(\Id+\tfrac12\lambda b_{\bfx}^{1/2}\bJ)^{-1}\,\left(
\bfw^0+\tfrac12\eps\lambda\,\bE_{\ubx}^{1/2}
\right)\,.
\]
Thus, as expected,
\[
\|\ubw^{1/2}\|\lesssim \frac{1}{\sqrt{1+\lambda^2}}+\eps
\]
and the claimed bound on $\|\ubw^{n+1/2}\|$ holds.

\subsection*{Bound on $\bfw^n$} For the next step of the proof the starting point is the $\bfw$ equation of \eqref{scheme:CN} written as
\[
\left(\Id-\tfrac12\lambda\,b_{\ubx}^{n+1/2}\,\bJ\right)\,\bfw^{n+1}
\,=\,
\left(\Id+\tfrac12\lambda\,b_{\ubx}^{n+1/2}\,\bJ\right)
\bfw^n
-\eps\,\lambda\,b_{\ubx}^{n+1/2}\,\bJ\bF_{\ubx}^{n+1/2}\,.
\]
Introducing for any $n\geq 1$
\[
\bfu^n:=\bfw^n-\eps\,\bJ\bF_{\ubx}^{n-1/2}
\]
one deduces that for any $n\geq 1$
\[
\left(\Id-\tfrac12\lambda\,b_{\ubx}^{n+1/2}\,\bJ\right)\,\bfu^{n+1}
\,=\,
\left(\Id+\tfrac12\lambda\,b_{\ubx}^{n+1/2}\,\bJ\right)
\bfu^n
-\eps\,\left(\Id+\tfrac12\lambda\,b_{\ubx}^{n+1/2}\,\bJ\right)(\bF_{\ubx}^{n+1/2}-\bF_{\ubx}^{n-1/2})
\,.
\]
so that 
\[
\|\bfu^{n+1}\|
\leq\|\bfu^{n}\|+\eps\,\|\bF_{\ubx}^{n+1/2}-\bF_{\ubx}^{n-1/2}\|\,.
\]

Benefiting from the bounds of the previous step we obtain for any $n\geq 1$ such that $n\Delta t\leq T$
\[
\|\bfu^{n+1}\|
-\|\bfu^{n}\|
\lesssim \Delta t\,\left(\frac{1}{\sqrt{1+\lambda^2}}+\eps\right)\,.
\]
By iteration this gives for any $n\geq 1$ such that $n\Delta t\leq T$
\begin{align*}
\|\bfw^n\|
&\lesssim \|\bfu^{n}\|+\eps
\lesssim \|\bfu^{1}\|+\eps+n\Delta t\,\left(\frac{1}{\sqrt{1+\lambda^2}}+\eps\right)\\
&\lesssim \|\bfw^{1}\|+\frac{1}{\sqrt{1+\lambda^2}}+\eps
\lesssim \|\bfw^{0}\|+\frac{1}{\sqrt{1+\lambda^2}}+\eps\lesssim 1\,,
\end{align*}
since $\bfw^1=\bfw^0+2\,\ubw^{1/2}$.

\subsection*{Bounds on $\bfx^{n+1}-\bfx^n$ and $\bfw^{n+1}-\bfw^n$} The last step follows readily from 
\begin{align*}
\bfx^{n+1}-\bfx^n&\,=\,\eps\,\lambda\,\ubw^{n+1/2}
\,-\,\Delta t\,\left(\ue^{n+1/2} \,-\, \dfrac{1}{2} \| \ubw^{n+1/2} \|^{2}\right)\nabla_{\bfx}^{\perp}\left(\dfrac1b\right)(\ubx^{n+1/2})\\
\eps\,(\bfw^{n+1}-\bfw^n)&\,=\,-\eps\,\lambda\,b_{\ubx}^{n+1/2}\,\bJ\,\ubw^{n+1/2}
\,+\,\Delta t\,\bE_{\ubx}^{n+1/2}
\end{align*}
and the fact that when $0\leq n\Delta t\leq T$
\[
\eps\,\lambda\,\|\ubw^{n+1/2}\|
\lesssim \eps\frac{\lambda}{\sqrt{1+\lambda^2}}+\eps^2\lambda \lesssim \min\left(\left\{\dfrac{\Delta t}{\eps};\eps\right\}\right)+\Delta t\,. 
\]

\section{Asymptotic analysis}\label{s:disc-x}

The goal of the present section is to provide a discrete counterpart to Proposition~\ref{p:cont-x}. With uniform bounds in hands, the main ingredients to add are multilinear elimination lemmas.

To prepare those we first gather bilinear discrete differentiation and averaging rules
$$
\left\{
  \begin{array}{l}
\ds  (ab)^{n+1}-(ab)^n\,=\,\ua^{n+1/2}\,(b^{n+1}-b^n)\,+\,(a^{n+1}-a^n)\,\ub^{n+1/2}\,,
  \\[0.8em]
\ds \left(\underline{ab} \right)^{n+1/2}\,=\,\ua^{n+1/2}\,\ub^{n+1/2}\,
+\,\frac14\,(a^{n+1}-a^n)\,(b^{n+1}-b^n)\,.
\end{array}\right.
$$

From \eqref{scheme:CN} stems the following discrete counterpart to \eqref{eq:velocity}
\begin{align} \label{scheme:velocity}
\ubw^{n+1/2}\, = \, \dfrac{\eps^{2}}{\Delta t}\, \dfrac{(\bfw^{n+1})^\perp-(\bfw^{n})^\perp}{b_{\ubx}^{n+1/2}}\,-\, \eps \,\left(\dfrac{\bE^{\perp}}{b}\right)_{\ubx}^{n+1/2}\,.
\end{align}
In turn, straightforward computations imply the following discrete counterparts to Lemmas~\ref{le:cont-1st} and~\ref{le:cont-2nd}.

\begin{lemma} \label{le:disc-1st}
Consider $\cL$ a differentiable function of space with values in linear maps from $\R^2$ to some $\R^n$, $n\in\N^*$. Then solutions to \eqref{scheme:CN} satisfy
\begin{align*}
(\cL_{\ubx}(\ubw))^{n+1/2} 
&\, = \,
-\eps^{2}\dfrac{(\kappa_{\cL}(\bfx,\bfw))^{n+1}-(\kappa_{\cL}(\bfx,\bfw))^{n}}{\Delta t}
\,+\,\eps\,(\eta_{\cL}(\ubx,\ue,\ubw))^{n+1/2}\\
&\quad
\,+\,\dfrac{\eps^2}{\Delta t}\,\tau_{\cL}(\bfx^n,\bfx^{n+1},\bfw^n,\bfw^{n+1})\,,
\end{align*}
where  $\kappa_{\cL}$ is given by
\begin{align*} 
\kappa_{\cL}(\bfx, \bfv) &
                           \,:=\,-\dfrac{1}{b(\bfx)}\cL(\bfx)(\bfv^{\perp})\,,
                           \end{align*}
whereas the couple $\left(\eta_{\cL}, \tau_{\cL}\right)$ is defined as
\begin{align*} 
\eta_{\cL}(\bfx,e,\bfv) &\,:=\,-\,\left(\dD_{\bfx}\left(\dfrac{\cL}{b}\right)(\bfx)(\bfv)\right)\,(\bfv^{\perp})
-\dfrac{1}{b(\bfx)}\,\cL(\bfx)(\bE^{\perp}(\bfx))\\
&\quad
\,+\,\eps\,\left(e-\tfrac12\|\bfv\|^{2}\right)\,\left(\dD_{\bfx}\left(\dfrac{\cL}{b}\right)(\bfx)\left(\nabla_{\bfx}^\perp\left(\frac1b\right)(\bfx)\right)\right)(\bfv^{\perp})\,,\\
\end{align*}
and 
\begin{align*}
  \tau_{\cL}(\bfx,\bfx',\bfv,\bfv') & \,:=\,
-\,\left(\dfrac{\cL}{b}(\bfx')-\dfrac{\cL}{b}(\bfx)
-\dD_{\bfx}\left(\dfrac{\cL}{b}\right)\left(\dfrac{\bfx'+\bfx}{2}\right)(\bfx'-\bfx)\right)\,\left(\left(\dfrac{\bfv'+\bfv}{2}\right)^\perp\right)\\
&\quad
\,-\,
\left(\frac12\left(\dfrac{\cL}{b}(\bfx')+\dfrac{\cL}{b}(\bfx)\right)-\dfrac{\cL}{b}\left(\dfrac{\bfx'+\bfx}{2}\right)\right)
\left((\bfv')^\perp-\bfv^\perp\right)\,.
\end{align*}
\end{lemma}

\begin{lemma} \label{le:disc-2nd}
Consider $\cB$ a differentiable function of space with values in bilinear maps from $\R^2\times\R^2$ to some $\R^n$, $n\in\N^*$. Then solutions to \eqref{scheme:CN} satisfy
\begin{align*}
(\cB_{\ubx}(\ubw,\ubw))^{n+1/2} 
&\, = \,\left(\frac12\|\ubw\|^2\,\Tr(\cB_\ubx)\right)^{n+1/2}
-\eps^{2}\dfrac{(\kappa_{\cB}(\bfx,\bfw))^{n+1}-(\kappa_{\cB}(\bfx,\bfw))^{n}}{\Delta t}
\,+\,\eps\,(\eta_{\cB}(\ubx,\ue,\ubw))^{n+1/2}\\
&\quad
\,+\,\dfrac{\eps^2}{\Delta t}\,\tau_{\cB}(\bfx^n,\bfx^{n+1},\bfw^n,\bfw^{n+1})\,,
\end{align*}
where
\begin{align*} 
\kappa_{\cB}(\bfx, \bfv) &\,:=\,-\frac12\,\dfrac{1}{b(\bfx)}\cB(\bfx)(\bfv,\bfv^{\perp})\,,\\
\end{align*}
whereas  $\eta_{\cB}$ is given as
\begin{align*}
  \eta_{\cB}(\bfx,e,\bfv) &\,:=\, -\,\frac12\,\left(\dD_{\bfx}\left(\dfrac{\cB}{b}\right)(\bfx)(\bfv)\right)\,(\bfv,\bfv^{\perp})
-\frac12\dfrac{1}{b(\bfx)}\,\left(\cB(\bfx)(\bfv,\bE^{\perp}(\bfx))+\cB(\bfx)(\bE(\bfx),\bfv^{\perp})\right)\\
&\qquad
+\frac12\,\eps\,\left(e-\tfrac12\|\bfv\|^{2}\right)\,\left(\dD_{\bfx}\left(\dfrac{\cB}{b}\right)(\bfx)\left(\nabla_{\bfx}^\perp\left(\frac1b\right)(\bfx)\right)\right)(\bfv,\bfv^{\perp})\,,\\
\end{align*}
and
\begin{align*}
  \tau_{\cB}(\bfx,\bfx',\bfv,\bfv') & \,:=\,
-\,\frac12\,\left(\dfrac{\cB}{b}(\bfx')-\dfrac{\cB}{b}(\bfx)
-\dD_{\bfx}\left(\dfrac{\cB}{b}\right)\left(\dfrac{\bfx'+\bfx}{2}\right)(\bfx'-\bfx)\right)\,\left(\dfrac{\bfv'+\bfv}{2},\left(\dfrac{\bfv'+\bfv}{2}\right)^\perp\right)\\
&\qquad
\,-\,
\frac12\left(\frac12\left(\dfrac{\cB}{b}(\bfx')+\dfrac{\cB}{b}(\bfx)\right)-\dfrac{\cB}{b}\left(\dfrac{\bfx'+\bfx}{2}\right)\right)\left(\dfrac{\bfv'+\bfv}{2},(\bfv')^\perp-\bfv^\perp\right)\\
&\qquad
\,-\,
\frac12\left(\frac12\left(\dfrac{\cB}{b}(\bfx')+\dfrac{\cB}{b}(\bfx)\right)-\dfrac{\cB}{b}\left(\dfrac{\bfx'+\bfx}{2}\right)\right)
\left(\bfv'-\bfv,\left(\dfrac{\bfv'+\bfv}{2}\right)^\perp\right)\\
&\qquad
\,-\,\frac18\,
\left(\dfrac{\cB}{b}(\bfx')-\dfrac{\cB}{b}(\bfx)\right)\left(\bfv'-\bfv,(\bfv')^\perp-\bfv^\perp\right)\,.
\end{align*}
\end{lemma}

Here we apply the foregoing lemmas with exactly the same choices as in the proof of Proposition~\ref{p:cont-x}, explicitly $\cL_1$ constant with value the identity map and $\cB_1$ defined by 
\[
\cB_1(\bfx)(\bfv_1,\bfv_2):=
-\dD_{\bfx}\left(\dfrac1b\right)(\bfx)(\bfv_1)\,\bfv_2^{\perp}\,.
\]

This allows us to transform the $\bfx$ equation of \eqref{scheme:CN} into
\begin{align}\label{eq:disc-x2}
\dfrac{(\bfx+\eps\,\kappa_1)^{n+1}-(\bfx+\eps\,\kappa_1)^{n}}{\Delta t}
&\,=\, \left(\bF_{\ubx}\,-\,\ue \left(\nabla_{\bfx}^{\perp}\left(\dfrac1b\right)\right)_{\ubx}\right)^{n+1/2}
+\,\eps\,(\eta_1^n+\tau_1^n)\,,
\end{align}
with for any $n$
\begin{align*}
\|\kappa_1^n\|&\lesssim
\|\bfw^n\|\,(1+\eps\,\|\bfw^n\|)\\[0.8em]
\|\eta_1^n\|&\lesssim \|\ubw^{n+1/2}\|
+\|\ubw^{n+1/2}\|\,(1+\eps\,\|\ubw^{n+1/2}\|)\,(|\ue^{n+1/2}|+\|\ubw^{n+1/2}\|^2)\\[0.8em]
\|\tau_1^n\|&\lesssim
\dfrac{\|\bfx^{n+1}-\bfx^n\|}{\Delta t}
\Big(
\|\bfx^{n+1}-\bfx^n\|^2\,\|\ubw^{n+1/2}\|\,(1+\eps\,\|\ubw^{n+1/2}\|)\\[0.8em]
&\qquad
+\|\bfx^{n+1}-\bfx^n\|\,\|\bfw^{n+1}-\bfw^n\|\,(1+\eps\,\|\ubw^{n+1/2}\|)
+\eps\,\|\bfw^{n+1}-\bfw^n\|^2
\Big)\,.
\end{align*}

This is sufficient to prove the following discrete counterpart to Proposition~\ref{p:cont-x}.

\begin{proposition}\label{p:disc-x}
Let $T>0$, $M>0$, $\eps_0>0$. There exist $\delta_0>0$ and $C>0$ such that when $0<\eps\leq\eps_0$ and $0<\Delta t\leq \delta_0$ if $(\bfx,e,\bfw)$ solves \eqref{scheme:CN} with initial datum $(\bfx^0,e^0,\bfw^0)$ such that $\|\bfv^0\|\leq M$, $|e^0|\leq \tfrac12M^2$ and $(\bfy,g)$ solves \eqref{scheme:asymp} with initial datum $(\bfy^0,g^0)=(\bfx^0,e^0)$ then when $0\leq n\,\Delta t\leq T$,
\begin{align*}
\|(\bfx^n,e^n)-(\bfy^n,g^n)\|\,\leq\,C\,\eps\,.
\end{align*}
\end{proposition}

\begin{proof}
To begin with we invoke Proposition~\ref{p:unif} to conclude that when $0\leq n\,\Delta t\leq T$ and $\Delta t$ is sufficiently small (independently of $\eps$), $\kappa_1^n$ and $\eta_1^n$ are uniformly bounded whereas
\[
\|\tau_1^n\|\lesssim \min\left(\left\{\dfrac{(\Delta t+\eps)^2}{\Delta t}; \dfrac{\Delta t}{\eps^2}\right\}\right)
\lesssim \min\left(\left\{\Delta t+\dfrac{1}{\lambda};\lambda\right\}\right)\lesssim 1+\Delta t\,,
\] 
where again $\lambda=\Delta t/\eps^2$.

One may then proceed as in the proof of Proposition~\ref{p:cont-x}, summing \eqref{eq:disc-x2} instead of integrating \eqref{eq:cont-x2}, to deduce that for some constant $K$ (depending on $(T,M,\eps_0)$ and the fields but not on $(\eps,\Delta t)$), provided that $0\leq n\,\Delta t\leq T$ and $\Delta t$ is sufficiently small (independently of $\eps$),
\[
\|\bfx^n-\bfy^n\|
\leq K\,\Delta t\,\sum_{\ell=1}^n\,\|\bfx^\ell-\bfy^\ell\|\,+\,K\,\eps
\,,
\]
which implies provided that $\Delta t$ is small enough to enforce $K\,\Delta t\leq \tfrac12$,
\[
\|\bfx^n-\bfy^n\|
\leq 2\,K\,\Delta t\,\sum_{\ell=1}^{n-1}\,\|\bfx^\ell-\bfy^\ell\|\,+\,2\,K\,\eps
\,,
\]

One then concludes with a discrete Gr\"onwall lemma. Explicitly, set
\[
\xi^n\,:=\,\sum_{\ell=1}^{n-1}\,\|\bfx^\ell-\bfy^\ell\|\,+\,\dfrac{\eps}{\Delta t}\,
\]
so that for any $n$ such that $0\leq n\,\Delta t\leq T$
\[
\xi^{n+1}\leq (1+2 K\,\Delta t)\, \xi^n\,.
\]
This implies recursively that for any $n$ such that $0\leq n\,\Delta t\leq T$
\[
\xi^n \leq \eD^{2K\,n\Delta t}\,\xi^0\,\leq\,\eD^{2\,K\,T}\,\dfrac{\eps}{\Delta t}
\,.
\]
In turn one deduces that for any $n$ such that $0\leq n\,\Delta t\leq T$
\[
\|\bfx^n-\bfy^n\|\,\leq\,2K\Delta t\,\xi^n\,\leq\,\eps\,2K\,\eD^{2\,K\,T}
\,.
\]
Since $e^n-g^n=-(\phi(\bfx^n)-\phi(\bfy^n))$, this concludes the proof with $C=2K\,\eD^{2\,K\,T}\,(1+\|\nabla_{\bfx}\phi\|_{L^\infty})$.
\end{proof}

\section{Direct convergence analysis}\label{s:convergence}

Before diving into the direct convergence analysis a few comments are in order. The general strategy follows the classical path of analyzing separately consistency errors and stability estimates. For background on classical numerical analysis of ordinary differential equations we refer for instance to \cite{HNW-book,HW-book}.

On the consistency side the basic observations here are that for any smooth function of time $a$
\begin{align*}
\dfrac{a(t+\Delta t)+a(t)}{2}-a\left(t+\tfrac12\Delta t\right)
&\,=\,\frac18\,(\Delta t)^2\int_{-1}^1 \left(\dfrac{\dD}{\dD t}\right)^2a\left(t+\tfrac12(1+\tau)\Delta t\right)\,(1-|\tau|)\dD \tau\\
\dfrac{a(t+\Delta t)-a(t)}{\Delta t}-\dfrac{\dD a}{\dD t}\left(t+\tfrac12\Delta t\right)
&\,=\,\frac{1}{16}\,(\Delta t)^2\int_{-1}^1 \left(\dfrac{\dD}{\dD t}\right)^3a\left(t+\tfrac12(1+\tau)\Delta t\right)\,(1-|\tau|)^2\dD \tau\,.
\end{align*}
Thus if $a$ is solving 
\[
\dfrac{\dD a}{\dD t}\,=\,\cV(a)
\]
for some vector field $\cV$ then
\[
\left\|
\dfrac{a(t+\Delta t)-a(t)}{\Delta t}-\cV\left(\dfrac{a(t+\Delta t)+a(t)}{2}\right)
\right\|
\leq \frac{(\Delta t)^2}{8}\left(\|\nabla_a\cV\|_{L^\infty}\,\left\|\left(\dfrac{\dD}{\dD t}\right)^2a\right\|_{L^\infty}
+\frac13\,\left\|\left(\dfrac{\dD}{\dD t}\right)^3a\right\|_{L^\infty}\right).
\]
This rough classical computation is sufficient to justify that we are indeed studying a second-order scheme and if we were not interested in the dependency on $\eps$ or if we were studying the regime when $\eps$ is bounded away from zero it would be sufficient to conclude.

It is also sufficient to predict what are the best consistency estimates we can expect. Indeed there is indeed no way, even through a near identity change of variable, to improve on the fact that the $k$-th derivative of $\bfv$ scales like $\eps^{-2k}$ so that the consistency error for $\bfv$ cannot be better than $(\Delta t)^2\,\eps^{-6}$. In turn this implies that the consistency error for a slow variable coupled with $\eps^m\bfv$ for some $m$ cannot be improved beyond $(\Delta t)^2\,\eps^{m-6}$. In particular the consistency error for $\bfx$ is capped by $(\Delta t)^2\,\eps^{-5}$ whereas the one for $\GC$ cannot beat $(\Delta t)^2\,\eps^{-4}$. For a more general discussion we again refer to \cite[Appendix~A]{FiRo21}. In particular, for the sake of comparison note that when \cite{FiRo21} analyzes a first-order scheme the corresponding ceilings are respectively $(\Delta t)\,\eps^{-3}$ and $(\Delta t)\,\eps^{-2}$.

On the stability side note that the implicitation is designed to improve how perturbations on fast variables impact both slow and fast variables. The main remaining limitation then arises from how slow variables affect fast variables. The corresponding Lipschitz constant is mostly due to the fact that the amplitude of the magnetic field is non constant and scales as $\eps^{-2}$. As a consequence one cannot start the convergence analysis before a level of preparation comparable to the one achieved to prove Proposition~\ref{p:cont-GC} so as to enforce that the analysis is performed with slow variables and $\eps^2\bfv$ thus taming the slow-to-fast stiffness. For the sake of comparison we recall that \cite{FiRo21} analyzes the case when $b$ is constant so that the slow-to-fast Lipschitz constant scales as $\eps^{-1}$ and the analysis may be carried out with slow variables and $\eps\,\bfv$, requiring only a cheaper preparation, similar to the one to prove Propositions~\ref{p:cont-x} and~\ref{p:disc-x}.

Let us stress that this means that even if one is only interested in $\bfx$ (and not in $\GC$) heeding to stability considerations one must nevertheless perform a higher order preparation. This is the reason why we have expounded Proposition~\ref{p:cont-GC} and its proof even if we omit to state and prove its discrete counterpart.

\subsection{Stability analysis}

As in Section~\ref{s:cont-gc} we introduce a guiding center variable
\[
\GC:=\bfx-\,\eps\,\dfrac{\bfw^{\perp}}{b(\bfx)}
\]
and the associated modified kinetic energy
\[
\eGC:=e\,+\,\dfrac{\eps}{b(\bfx)}\,\bE^\perp(\bfx)\cdot\bfw\,.
\]
Note that with $\kappa_e$ as in \eqref{eq:cont-e}
\begin{align}\label{eq:disc-e}
\eGC^{n+1}+\phi(\GC^{n+1})+\eps^2\,\kappa_e(\bfx^{n+1},\bfw^{n+1})
\,=\,\eGC^{n}+\phi(\GC^{n})+\eps^2\,\kappa_e(\bfx^{n},\bfw^{n})\,.
\end{align}

To proceed we need both a discrete version of Lemma~\ref{le:cont-3rd} and an extension of Lemma~\ref{le:disc-1st} allowing for linear operators depending on more slow variables. To facilitate the derivation of the latter note that solutions to \eqref{scheme:CN} satisfy
\begin{align*}
\dfrac{\GC^{n+1}-\GC^{n}}{\Delta t}
&\,=\,\,-\,\left(\left(\ue-\tfrac12\|\ubw\|^2\right)\left(\nabla_{\bfx}^{\perp}\left(\dfrac1b\right)\right)_{\ubx}\right)^{n+1/2}+(\eta_{\cL_1}(\ubx,\ue,\ubw))^{n+1/2}\\
&\qquad
\,+\,\dfrac{\eps}{\Delta t}\,\tau_{\cL_1}(\bfx^n,\bfx^{n+1},\bfw^n,\bfw^{n+1})\,,
\end{align*}
with notation from Lemma~\ref{le:disc-1st} and $\cL_1\equiv \Id$.

\begin{lemma} \label{le:disc-3rd}
Consider $\cT$ a differentiable function of space with values in symmetric trilinear maps from $\R^2\times\R^2\times\R^2$ to some $\R^n$, $n\in\N^*$. Then solutions to \eqref{scheme:CN} satisfy
\begin{align*}
(\cT_{\ubx}(\ubw,\ubw,\ubw))^{n+1/2} 
&\, = \,
-\eps^{2}\dfrac{(\kappa_{\cT}(\bfx,\bfw))^{n+1}-(\kappa_{\cT}(\bfx,\bfw))^{n}}{\Delta t}
\,+\,\eps\,(\eta_{\cT}(\ubx,\ue,\ubw))^{n+1/2}\\
&\quad
\,+\,\dfrac{\eps^2}{\Delta t}\,\tau_{\cT}(\bfx^n,\bfx^{n+1},\bfw^n,\bfw^{n+1})\,,
\end{align*}
where
\begin{align*} 
\kappa_{\cT}(\bfx, \bfv) &\,:=\,-\,
\dfrac{1}{b(\bfx)}\left(\cT(\bfx)(\bfv,\bfv,\bfv^{\perp})+\frac23\,\cT(\bfx)(\bfv^\perp,\bfv^\perp,\bfv^{\perp})\right)\,,
\end{align*}
and
\begin{align*}
\eta_{\cT}(\bfx,e,\bfv) &\,:=\, -\,\left(\dD_{\bfx}\left(\dfrac{\cT}{b}\right)(\bfx)(\bfv)\right)\,(\bfv,\bfv,\bfv^{\perp})
-\frac23\,\left(\dD_{\bfx}\left(\dfrac{\cT}{b}\right)(\bfx)(\bfv)\right)(\bfv^\perp,\bfv^\perp,\bfv^{\perp})\\
&\qquad
-\dfrac{1}{b(\bfx)}\,\left(\cT(\bfx)(\bfv,\bfv,\bE^{\perp}(\bfx))+2\cT(\bfx)(\bE(\bfx),\bfv,\bfv^{\perp})
+2\cT(\bfx)(\bfv^\perp,\bfv^\perp,\bE^{\perp}(\bfx))\right)\\
&\qquad
+\,\eps\,\left(e-\tfrac12\|\bfv\|^{2}\right)\,\left(\dD_{\bfx}\left(\dfrac{\cT}{b}\right)(\bfx)\left(\nabla_{\bfx}^\perp\left(\frac1b\right)(\bfx)\right)\right)(\bfv,\bfv,\bfv^{\perp})\\
&\qquad
+\,\frac23\,\eps\,\left(e-\tfrac12\|\bfv\|^{2}\right)\,\left(\dD_{\bfx}\left(\dfrac{\cT}{b}\right)(\bfx)\left(\nabla_{\bfx}^\perp\left(\frac1b\right)(\bfx)\right)\right)(\bfv^\perp,\bfv^\perp,\bfv^{\perp})\,,
\end{align*}
\begin{align*}
\tau_{\cT}&(\bfx,\bfx',\bfv,\bfv')\,:=\,
-\,\left(\dfrac{\cT}{b}(\bfx')-\dfrac{\cT}{b}(\bfx)
-\dD_{\bfx}\left(\dfrac{\cT}{b}\right)\left(\dfrac{\bfx'+\bfx}{2}\right)(\bfx'-\bfx)\right)\,\left(\dfrac{\bfv'+\bfv}{2},\dfrac{\bfv'+\bfv}{2},\left(\dfrac{\bfv'+\bfv}{2}\right)^\perp\right)\\
&\quad
-\,\frac23\left(\dfrac{\cT}{b}(\bfx')-\dfrac{\cT}{b}(\bfx)
-\dD_{\bfx}\left(\dfrac{\cT}{b}\right)\left(\dfrac{\bfx'+\bfx}{2}\right)(\bfx'-\bfx)\right)\,\left(\left(\dfrac{\bfv'+\bfv}{2}\right)^\perp,\left(\dfrac{\bfv'+\bfv}{2}\right)^\perp,\left(\dfrac{\bfv'+\bfv}{2}\right)^\perp\right)\\
&\quad
\,-\,
\left(\frac12\left(\dfrac{\cT}{b}(\bfx')+\dfrac{\cT}{b}(\bfx)\right)-\dfrac{\cT}{b}\left(\dfrac{\bfx'+\bfx}{2}\right)\right)
\left(\dfrac{\bfv'+\bfv}{2},\dfrac{\bfv'+\bfv}{2},\left(\bfv'\right)^\perp-\bfv^\perp\right)\\
&\quad
\,-\,
2\left(\frac12\left(\dfrac{\cT}{b}(\bfx')+\dfrac{\cT}{b}(\bfx)\right)-\dfrac{\cT}{b}\left(\dfrac{\bfx'+\bfx}{2}\right)\right)
\left(\bfv'-\bfv,\dfrac{\bfv'+\bfv}{2},\left(\dfrac{\bfv'+\bfv}{2}\right)^\perp\right)\\
&\quad
\,-\,
2\left(\frac12\left(\dfrac{\cT}{b}(\bfx')+\dfrac{\cT}{b}(\bfx)\right)-\dfrac{\cT}{b}\left(\dfrac{\bfx'+\bfx}{2}\right)\right)
\left(\left(\dfrac{\bfv'+\bfv}{2}\right)^\perp,\left(\dfrac{\bfv'+\bfv}{2}\right)^\perp,\left(\bfv'\right)^\perp-\bfv^\perp\right)\\
&\quad
\,-\,
\frac18\left(\dfrac{\cT}{b}(\bfx')+\dfrac{\cT}{b}(\bfx)\right)
\left(\bfv'-\bfv,\bfv'-\bfv,\left(\bfv'\right)^\perp-\bfv^\perp\right)\\
&\quad
\,-\,
\frac{1}{12}\left(\dfrac{\cT}{b}(\bfx')+\dfrac{\cT}{b}(\bfx)\right)
\left(\left(\bfv'\right)^\perp-\bfv^\perp,\left(\bfv'\right)^\perp-\bfv^\perp,\left(\bfv'\right)^\perp-\bfv^\perp\right)\\
&\quad
\,-\,\frac14\,
\left(\dfrac{\cT}{b}(\bfx')-\dfrac{\cT}{b}(\bfx)\right)\left(\bfv'-\bfv,\bfv'-\bfv,\left(\dfrac{\bfv'+\bfv}{2}\right)^\perp\right)\\
&\quad
\,-\,\frac12\,
\left(\dfrac{\cT}{b}(\bfx')-\dfrac{\cT}{b}(\bfx)\right)\left(\dfrac{\bfv'+\bfv}{2},\bfv'-\bfv,(\bfv')^\perp-\bfv^\perp\right)\\
&\quad
\,-\,\frac12\,
\left(\dfrac{\cT}{b}(\bfx')-\dfrac{\cT}{b}(\bfx)\right)\left(\left(\dfrac{\bfv'+\bfv}{2}\right)^\perp,(\bfv')^\perp-\bfv^\perp,(\bfv')^\perp-\bfv^\perp\right)\,.
\end{align*}
\end{lemma}

Again the proof of Lemma~\ref{le:disc-3rd} is a straightforward --- but lengthy ! --- computation, that we mostly omit. To provide some elements of it, let us only stress that if $\bT_0$ is a fixed symmetric trilinear map then
\begin{align*}
\left(\bT_0(a,a,b)\right)^{n+1}
&-\left(\bT_0(a,a,b)\right)^{n}\\
&\quad=\,\bT_0(\ua^{n+1/2},\ua^{n+1/2},b^{n+1}-b^n)
+\frac14\,\bT_0(a^{n+1}-a^n,a^{n+1}-a^n,b^{n+1}-b^n)\\
&\qquad+2\,\bT_0(a^{n+1}-a^n,\ua^{n+1/2},\ub^{n+1/2})
\end{align*}
and
\begin{align*}
\underline{\bT_0(a,a,b)}^{n+1/2}
&=\bT_0(\ua^{n+1/2},\ua^{n+1/2},\ub^{n+1/2})
+\frac14\bT_0(a^{n+1}-a^n,a^{n+1}-a^n,\ub^{n+1/2})\\
&\quad+\frac12\bT_0(\ua^{n+1/2},a^{n+1}-a^n,b^{n+1}-b^n)
\end{align*}
thus in particular
\begin{align*}
\left(\bT_0(a,a,a)\right)^{n+1}
&-\left(\bT_0(a,a,a)\right)^{n}\\
&\,=\,3\,\bT_0(\ua^{n+1/2},\ua^{n+1/2},a^{n+1}-a^n)
+\frac14\,\bT_0(a^{n+1}-a^n,a^{n+1}-a^n,a^{n+1}-a^n)
\end{align*}
and
\begin{align*}
\underline{\bT_0(a,a,a)}^{n+1/2}
&=\bT_0(\ua^{n+1/2},\ua^{n+1/2},\ua^{n+1/2})
+\frac34\bT_0(\ua^{n+1/2},a^{n+1}-a^n,a^{n+1}-a^n)\,.
\end{align*}

\begin{lemma} \label{le:disc-1st-bis}
Consider $\cL$ a differentiable function of space-guiding center-kinetic energy with values in linear maps from $\R^2$ to some $\R^n$, $n\in\N^*$. Then solutions to \eqref{scheme:CN} satisfy
\begin{align*}
(\cL_{\ubx,\uGC,\ube}(\ubw))^{n+1/2} 
&\, = \,
-\eps^{2}\dfrac{(\kappa_{\cL}(\bfx,\GC,e,\bfw))^{n+1}-(\kappa_{\cL}(\bfx,\GC,e,\bfw))^{n}}{\Delta t}
\,+\,\eps\,(\eta_{\cL}(\ubx,\uGC,\ue,\ubw))^{n+1/2}\\
&\quad
\,+\,\dfrac{\eps^2}{\Delta t}\,\tau_{\cL}(\bfx^n,\bfx^{n+1},\GC^n,\GC^{n+1},e^n,e^{n+1},\bfw^n,\bfw^{n+1})\,,
\end{align*}
where
\begin{align*} 
\kappa_{\cL}(\bfx,\bfy,e,\bfv) & \,:=\,-\dfrac{1}{b(\bfx)}\cL(\bfx,\bfy,e)(\bfv^{\perp})\,,
\end{align*}
and
\begin{align*}
\eta_{\cL}&(\bfx,\bfy,e,\bfv)\\
&\,:=\,-\,\left(\dD_{\bfx}\left(\dfrac{\cL}{b}\right)(\bfx,\bfy,e)(\bfv)\right)\,(\bfv^{\perp})
-\dfrac{1}{b(\bfx)}\,\cL(\bfx,\bfy,e)(\bE^{\perp}(\bfx))\\
&\quad
\,+\,\eps\,\left(e-\tfrac12\|\bfv\|^{2}\right)\,\left(\dD_{\bfx}\left(\dfrac{\cL}{b}\right)(\bfx,\bfy,e)\left(\nabla_{\bfx}^\perp\left(\frac1b\right)(\bfx)\right)\right)(\bfv^{\perp})\\
&\quad-\frac{1}{b(\bfx)}\,\left(\dD_e \cL(\bfx,\bfy,e)(\bE(\bfx)\cdot\bfv)\right)(\bfv^{\perp})\\
&\quad+\frac{\eps}{b(\bfx)}\,\left(e-\tfrac12\|\bfv\|^{2}\right)
\,\left(\dD_e \cL(\bfx,\bfy,e)\left(\bE(\bfx)\cdot\nabla_{\bfx}^\perp\left(\frac1b\right)(\bfx)\right)\right)(\bfv^{\perp})\\
&\quad
\,-\,\eps\,\left(\dD_{\bfy}\left(\dfrac{\cL}{b}\right)(\bfx,\bfy,e)\left(
\,-\,\left(e-\tfrac12\|\bfv\|^2\right)\,\nabla_{\bfx}^{\perp}\left(\dfrac1b\right)(\bfx)
+\eta_{\cL_1}(\bfx,e,\bfv)\right)
\right)(\bfv^{\perp})\,,
\end{align*}
\begin{align*}
\tau_{\cL}&(\bfx,\bfx',\bfy,\bfy',e,e',\bfv,\bfv')\\
&:=\,
-\,\Bigg(\dfrac{\cL}{b}(\bfx',\bfy',e')-\dfrac{\cL}{b}(\bfx,\bfy,e)\\
&\qquad\qquad\qquad
-\dD_{(\bfx,\bfy,e)}\left(\dfrac{\cL}{b}\right)\left(\dfrac{\bfx'+\bfx}{2},\dfrac{\bfy'+\bfy}{2},\dfrac{e'+e}{2}\right)((\bfx'-\bfx,\bfy'-\bfy,e'-e))\Bigg)\,\left(\left(\dfrac{\bfv'+\bfv}{2}\right)^\perp\right)\\
&\ 
\,-\,
\left(\frac12\left(\dfrac{\cL}{b}(\bfx',\bfy',e')+\dfrac{\cL}{b}(\bfx,\bfy,e)\right)-\dfrac{\cL}{b}\left(\dfrac{\bfx'+\bfx}{2},\dfrac{\bfy'+\bfy}{2},\dfrac{e'+e}{2}\right)\right)
\left((\bfv')^\perp-\bfv^\perp\right)\\
&\ 
\,+\,\left(\dD_e\left(\dfrac{\cL}{b}\right)\left(\dfrac{\bfx'+\bfx}{2},\dfrac{\bfy'+\bfy}{2},\dfrac{e'+e}{2}\right)\left(\phi(\bfx')-\phi(\bfx)-\dD_{\bfx}\phi\left(\dfrac{\bfx'+\bfx}{2}\right)(\bfx'-\bfx)\right)
\right)\left(\left(\dfrac{\bfv'+\bfv}{2}\right)^\perp\right)\\
&\ 
\,-\,\left(\dD_{\bfy}\left(\dfrac{\cL}{b}\right)\left(\dfrac{\bfx'+\bfx}{2},\dfrac{\bfy'+\bfy}{2},\dfrac{e'+e}{2}\right)\left(
\tau_{\cL_1}(\bfx,\bfx',\bfv,\bfv')\right)
\right)\left(\left(\dfrac{\bfv'+\bfv}{2}\right)^\perp\right)\,.
\end{align*}
\end{lemma}

The proof of Lemma~\ref{le:disc-1st-bis} is a straightforward variation on the proof of Lemma~\ref{le:disc-1st}.

Then we apply Lemma~\ref{le:disc-3rd} to $\cT_1$ and Lemma~\ref{le:disc-1st-bis} to $\cL_2$ with  $\cT_1$ and $\cL_2$ defined as in Subsection~\ref{s:cont-gc} so as to modify \eqref{eq:disc-x2}. With the consistency analysis in mind, we want to express the outcome, at leading order, as a scheme for \eqref{eq:cont-gc}. To do so we need to compare \eqref{eq:disc-x2} with \eqref{eq:cont-x2} and keep track of differences between Lemmas~\ref{le:disc-3rd} and~\ref{le:disc-1st-bis} and Lemmas~\ref{le:cont-3rd} and~\ref{le:cont-1st}.

To begin with note that with $\kappa_1$ and $\eta_1$ as in \eqref{eq:cont-x2}, equation \eqref{eq:disc-x2} takes the form 
\begin{align*}
&\dfrac{(\bfx+\eps\,\kappa_1(\bfx,\bfw))^{n+1}-(\bfx+\eps\,\kappa_1(\bfx,\bfw))^{n}}{\Delta t}\\
&\qquad
\,=\, \left(\bF_{\ubx}\,-\,\ue \left(\nabla_{\bfx}^{\perp}\left(\dfrac1b\right)\right)_{\ubx}
+\eps\,\eta_1(\ubx,\ubw)\right)^{n+1/2}
+\,\eps\,\left(\tau_1^n+\ttau_1^n\right)\,,
\end{align*}
with $\tau_1^n$ as in \eqref{eq:disc-x2} and
\[
\ttau_1^n
\,:=\,\frac12\,\left(\left(\ue-\tfrac12\|\ubw\|^{2}\right)\,\left(\dD_{\bfx}\left(\dfrac{\cB_1}{b}\right)(\ubx)\left(\nabla_{\bfx}^\perp\left(\frac1b\right)(\ubx)\right)\right)(\ubw,\ubw^{\perp})\right)^{n+1/2}\,.
\]

Then with $\kappa_2$ and $\eta_3$ as in \eqref{eq:cont-gc}, equation \eqref{eq:disc-x2} is transformed into
\begin{align}\label{eq:disc-gc}
&\dfrac{(\GC+\eps^2\,\kappa_2(\bfx,\GC,e,\bfw))^{n+1}-(\GC+\eps^2\,\kappa_2(\bfx,\GC,e,\bfw))^{n}}{\Delta t}
\\\nonumber
&
\ =\left(\bF_{\uGC}\,-\,\ueGC \left(\nabla_{\bfx}^{\perp}\left(\dfrac1b\right)\right)_{\uGC}
+\eps^2\,\eta_2(\ubx,\GC,e,\ubw)\right)^{n+1/2}
+\,\eps\,\left(\tau_1^n+\ttau_1^n\right)
+\,\eps^2\,\left(\tau_2^n+\ttau_2^n\right)\,,
\nonumber
\end{align}
with
\[
\tau_2^n\,:=\,
\dfrac{\eps}{\Delta t}\,\tau_{\cL_2}(\bfx^n,\bfx^{n+1},\GC^n,\GC^{n+1},e^n,e^{n+1},\bfw^n,\bfw^{n+1})
\,+\,\dfrac{\eps}{\Delta t}\,\tau_{\cT_1}(\bfx^n,\bfx^{n+1},\bfw^n,\bfw^{n+1})
\]
and
\begin{align*}
\ttau_2^n
&\,:=\,\left(
\left(\ube-\tfrac12\|\ubw\|^{2}\right)\,\left(\dD_{\bfx}\left(\dfrac{\cT_1}{b}\right)(\ubx)\left(\nabla_{\bfx}^\perp\left(\frac1b\right)(\ubx)\right)\right)(\ubw,\ubw,\ubw^{\perp})
\right)^{n+1/2}\\
&\quad
\,+\,\frac23
\left(
\left(\ube-\tfrac12\|\ubw\|^{2}\right)\,\left(\dD_{\bfx}\left(\dfrac{\cT_1}{b}\right)(\bfx)\left(\nabla_{\bfx}^\perp\left(\frac1b\right)(\ubx)\right)\right)(\ubw^\perp,\ubw^\perp,\ubw^{\perp})
\right)^{n+1/2}\\
&\quad
\,+\,
\left(
\left(\ube-\tfrac12\|\ubw\|^{2}\right)\,\left(
\left(\dD_{\bfx}\left(\dfrac{\cL_2}{b}\right)+
\dD_{\bfy}\left(\dfrac{\cL_2}{b}\right)
\right)(\ubx,\uGC,\ue)\left(\nabla_{\bfx}^\perp\left(\frac1b\right)(\ubx)\right)\right)(\ubw^{\perp})
\right)^{n+1/2}\\
&\quad
\,+\,
\left(
\left(\ube-\tfrac12\|\ubw\|^{2}\right)\,\frac{1}{b(\ubx)}
\left(\dD_{e}\cL_2(\ubx,\uGC,\ue)\left(\bE(\ubx)\cdot\nabla_{\bfx}^\perp\left(\frac1b\right)(\ubx)
\right)\right)(\ubw^{\perp})
\right)^{n+1/2}
\,.
\end{align*}

Consistently, to any sequence $(\bfx,e,\bfw)$, (not necessarily solving \eqref{scheme:CN}) we associate $\GC$, $\eGC$ through 
\begin{align*}
\GC&:=\bfx-\,\eps\,\dfrac{\bfw^{\perp}}{b(\bfx)}\,,&
\eGC&:=e\,-\,\dfrac{\eps}{b(\bfx)}\,\bE^\perp(\bfx)\cdot\bfw\,,
\end{align*}
and set 
\begin{align*}
\Res_x^n(\bfx,e,\bfw)&:=
\dfrac{(\GC+\eps^2\,\kappa_2(\bfx,\GC,e,\bfw))^{n+1}-(\GC+\eps\,\kappa_2(\bfx,\GC,e,\bfw))^{n}}{\Delta t}\\
&\qquad\qquad\,-\,\left(\bF_{\uGC}\,-\,\ueGC \left(\nabla_{\bfx}^{\perp}\left(\dfrac1b\right)\right)_{\uGC}
+\eps^2\,\eta_2(\ubx,\GC,e,\ubw)\right)^{n+1/2}
\end{align*}
\begin{align*}
\Res_e^n(\bfx,e,\bfw)&:=
\dfrac{(\eGC+\phi(\GC)+\eps^2\,\kappa_e(\bfx,\bfw))^{n+1}-(\eGC+\phi(\GC)+\eps^2\,\kappa_e(\bfx,\bfw))^{n}}{\Delta t}
\end{align*}
and
\begin{align*}
\Res_w^n(\bfx,e,\bfw)&:=
\dfrac{\bfw^{n+1}-\bfw^{n}}{\Delta t}
\,-\,\left(\eps^{-1}\bE_\ubx\,-\,\eps^{-2}\,b_\ubx\,\ubw^{\perp}
\right)^{n+1/2}\,.
\end{align*}

\begin{proposition}\label{p:stab}
Let $T>0$, $M'>0$. There exist $\eps_0>0$, $\delta_0>0$ and $C>0$ such that when $0<\eps\leq\eps_0$ and $0<\Delta t\leq \delta_0$ if $(\bfx,e,\bfw)$ and $(\tbx,\te,\tbw)$ are such that when $0\leq n\,\Delta t\leq T$,
\begin{align*}
\|\bfw^n\|&\leq M'\,,& |e^n|&\leq \tfrac12(M')^2\,,&
\|\tbw^n\|&\leq M'\,,& |\te^n|&\leq \tfrac12(M')^2\,,&
\end{align*}
then when $0\leq n\,\Delta t\leq T$,
\begin{align*}
\|(\GC^n,\eGC^n)&-(\tGC^n,\teGC^n)\|
\,+\,\eps\,\|(\bfx^n,e^n)-(\tbx^n,\te^n)\|
\,+\,\eps^2\,\|\bfw^n-\tbw^n\|\\
&\,\leq\,C\,\left(
\|(\GC^0,\eGC^0)-(\tGC^0,\teGC^0)\|
\,+\,\eps^2\,\|\bfw^0-\tbw^0\|
\right)\\
&\qquad\,+\,C\,
\max_{0\leq k\leq n}\left(\|\Res_x^k(\bfx,e,\bfw)\|
+\|\Res_x^k(\tbx,\te,\tbw)\|
+|\Res_e^k(\bfx,e,\bfw)|
+|\Res_e^k(\tbx,\te,\tbw)|\right)\\
&\qquad\,+\,C\,\eps^2\,\max_{0\leq k\leq n}\left(\|\Res_w^k(\bfx,e,\bfw)\|
+\eps^2\,\|\Res_w^k(\tbx,\te,\tbw)\|
\right)
\,.
\end{align*}
\end{proposition}

\begin{proof}
Note that when $\bfw$ is restricted to a bounded set the map $(\bfx,e,\bfw)\mapsto (\GC-\bfx,\eGC-e)$ has Lipschitz bound $\cO(\eps)$ so that when $\eps$ is small enough (depending on the bound on $\bfw$) by the Banach fixed point theorem one may invert the relation and see $\bfx$, $e$ as functions of $(\GC,\eGC,\bfw)$ with Lipschitz dependence encoded by
\[
\|(\bfx,e)-(\tbx,\te)\|
\lesssim 
\|(\GC,\eGC)-(\tGC,\teGC)\|\,+\,\eps\,\|\bfw-\tbw\|\,.
\] 
In particular it is sufficient to estimate  $\|(\GC^n,\eGC^n)-(\tGC^n,\teGC^n)\|\,+\,\eps^2\,\|\bfw^n-\tbw^n\|$.

Likewise, seeing $\bfx$, $e$ as functions of $(\GC,\eGC,\bfw)$ and restricting to bounded sets of $(\eGC,\bfw)$ the map $(\GC,\eGC,\bfw)\mapsto \eps^2(\kappa_2(\bfx,\GC,e,\bfw),\kappa_e(\bfx,\bfw))$ has Lipschitz bound $\cO(\eps^2)$ so that the relations
\begin{align*}
\yGC&:=\GC+\eps^2\kappa_2(\bfx,\GC,e,\bfw)\,,&
\gGC&:=\eGC+\eps^2\,\kappa_e(\bfx,e)\,,
\end{align*}
may be inverted with Lipschitz dependence
\[
\|(\GC,\eGC)-(\tGC,\teGC)\|
\lesssim 
\|(\yGC,\gGC)-(\tyGC,\tgGC)\|\,+\,\eps^2\,\|\bfw-\tbw\|\,.
\] 
In particular it is sufficient to estimate  $\|(\yGC^n,\gGC^n)-(\tyGC^n,\tgGC^n)\|\,+\,\eps^2\,\|\bfw^n-\tbw^n\|$.

With this in mind one may interpret $\Res$ terms as forcing consistency errors in schemes for $(\yGC,\gGC,\bfw)$ and $(\tyGC,\tgGC,\tbw)$ so that heeding to the $\eps$-dependence of equations for some constant $K$ (independent of $(\eps,\Delta t)$), when $0\leq n\,\Delta t\leq T$
\begin{align*}
\|(\yGC^{n+1},\gGC^{n+1})&-(\tyGC^{n+1},\tgGC^{n+1})\|
\,+\,\eps^2\,\|\bfw^{n+1}-\tbw^{n+1}\|\\
&\,\leq\,
(1\,+\,K \Delta t)
\left(\|(\yGC^n,\gGC^n)-(\tyGC^n,\tgGC^n)\|
\,+\,\eps^2\,\|\bfw^n-\tbw^n\|\right)\\
&\qquad
\,+\,K \Delta t
\left(
\|(\yGC^{n+1},\gGC^{n+1})-(\tyGC^{n+1},\tgGC^{n+1})\|
\,+\,\eps^2\,\|\bfw^{n+1}-\tbw^{n+1}\|
\right)\\
&\qquad\,+\,K \Delta t\,
\left(\|\Res_x^n(\bfx,e,\bfw)\|
+\|\Res_x^n(\tbx,\te,\tbw)\|
+|\Res_e^n(\bfx,e,\bfw)|
+|\Res_e^n(\tbx,\te,\tbw)|\right)\\
&\qquad\,+\,K \Delta t\,\eps^2\,\left(\|\Res_w^n(\bfx,e,\bfw)\|
+\eps^2\,\|\Res_w^n(\tbx,\te,\tbw)\|
\right)
\,.
\end{align*}
One then concludes as in the proof of Proposition~\ref{p:disc-x} by taking $\Delta t$ sufficiently small and invoking a discrete Gr\"onwall lemma.
\end{proof}

\subsection{Consistency analysis}

The goal is obviously to apply Proposition~\ref{p:stab} to compare a solution to \eqref{scheme:CN} with the time discretization of a solution to \eqref{eq:ODE}. On the latter side the discussion at the beginning of the current section already proves the following lemma.

\begin{lemma}\label{le:cont-cons}
Let $T>0$, $M>0$, $\eps_0>0$ and $\delta_0>0$. There exists $C>0$ such that when $0<\eps\leq\eps_0$ and $0<\Delta t\leq\delta_0$ if $(\bfx,\bfv)$ solves \eqref{eq:ODE} with initial datum $(\bfx^0,\bfv^0)$ such that $\|\bfv^0\|\leq M$ and $(\tbx,\te,\tbw)$ is defined by 
\begin{align*}
\tbx^n&:=\bfx(n\,\Delta t)\,,&
\te^n&:=\frac12\|\bfv(n\,\Delta t)\|^2\,,&
\tbw^n&:=\bfv(n\,\Delta t)\,,&
\end{align*}
then for any $n$ such that $0\leq n\Delta t\leq T$,
\[
\|\Res_x^n(\tbx,\te,\tbw)\|
+|\Res_e^n(\tbx,\te,\tbw)|
+\eps^2\,\|\Res_w^n(\tbx,\te,\tbw)\|
\leq C\,\dfrac{(\Delta t)^2}{\eps^4}\,.
\]
\end{lemma}

The following lemma completes the estimates of Proposition~\ref{p:unif} hence providing estimates necessary to apply Proposition~\ref{p:stab}. 

\begin{lemma}\label{le:disc-cons}
Let $T>0$, $M>0$ and $\eps_0>0$. There exist $\delta_0>0$ and $C>0$ such that when $0<\eps\leq\eps_0$ and $0<\Delta t\leq \delta_0$ if $(\bfx,e,\bfw)$ solves \eqref{scheme:CN} with initial datum $(\bfx^0,e^0,\bfw^0)$ such that $\|\bfw^0\|\leq M$, $e^0=\tfrac12\|\bfw^0\|^2$ then when $0\leq n\,\Delta t\leq T$,
\[
\left|\ \ue^{n+1/2}-\frac12\|\ubw^{n+1/2}\|^2\ \right|
\leq C\,\dfrac{(\Delta t)^2}{\eps^4}\,.
\]
\end{lemma}

\begin{proof}
To begin with, note that 
\[
\ue^{n+1/2}-\frac12\|\ubw^{n+1/2}\|^2
\,=\,\left(\underline{e-\frac12\|\bfw\|^2}\right)^{n+1/2}
\,+\,\frac18\,\|\bfw^{n+1}-\bfw^n\|^2
\]
so that in view of $\|\bfw^{n+1}-\bfw^n\|$ estimates from Proposition~\ref{p:unif} it is sufficient to bound $e-\frac12\|\bfw\|^2$.

Now recall that by taking the scalar product of the $\bfw$ equation of \eqref{scheme:CN} with $\ubw^{n+1/2}$ one obtains
\[
\,\|\bfw^{n+1}\|^2-\,\|\bfw^{n}\|^2
\,=\,2\,\Delta t\,\bE_{\ubx}^{n+1/2}\cdot \dfrac{\ubw^{n+1/2}}{\eps}\,,
\]
so that 
\begin{align*}
\left(e-\frac12\|\bfw\|^2\right)^{n+1}
-\left(e-\frac12\|\bfw\|^2\right)^n
&\,=\,-\Delta t\,\left(\phi(\bfx^{n+1})-\phi(\bfx^n)
-\dD_\bfx\phi(\ubx^{n+1/2})(\bfx^{n+1}-\bfx^n)\right)\\
&\qquad -\Delta t\,\left(\left(\ue-\frac12\|\ubw\|^2\right)\,\bE_{\ubx}\cdot\nabla_{\bfx}^\perp\left(\frac1b\right)(\ubx)\right)^{n+1/2}
\end{align*}
hence for some constant $K$ (independent of $(\eps,\Delta t)$)
\begin{align*}
\left|\left(e-\frac12\|\bfw\|^2\right)^{n+1}
-\left(e-\frac12\|\bfw\|^2\right)^n\right|
\,&\leq\,K\,\Delta t\,
\left|\left(\underline{e-\frac12\|\bfw\|^2}\right)^{n+1/2}\right|\\
&\quad+K\,\Delta t\,\left(\|\bfx^{n+1}-\bfx^n\|^2+\|\bfw^{n+1}-\bfw^n\|^2\right)\,.
\end{align*}
With estimates from Proposition~\ref{p:unif} one then concludes as in the proof of Proposition~\ref{p:disc-x} by taking $\Delta t$ sufficiently small and invoking a discrete Gr\"onwall lemma.
\end{proof}

\begin{corollary}\label{c:cons}
Let $T>0$, $M>0$, $\eps_0>0$. There exist $\delta_0>0$ and $C>0$ such that when $0<\eps\leq\eps_0$ and $0<\Delta t\leq \delta_0$ if $(\bfx,e,\bfw)$ solves \eqref{scheme:CN} with initial datum $(\bfx^0,e^0,\bfw^0)$ such that $\|\bfw^0\|\leq M$, $e^0=\tfrac12\|\bfw^0\|^2$ then when $0\leq n\,\Delta t\leq T$,
\[
\|\Res_x^n(\bfx,e,\bfw)\|
+|\Res_e^n(\bfx,e,\bfw)|
+\eps^2\,\|\Res_w^n(\bfx,e,\bfw)\|
\leq C\,\dfrac{(\Delta t)^2}{\eps^4}\,.
\]
\end{corollary}

\subsection{Convergence}

\begin{theorem}\label{th:conv}
Let $T>0$, $M>0$ and $\eps_0>0$. There exist $\delta_0>0$ and $C>0$ such that when $0<\eps\leq\eps_0$ and $0<\Delta t\leq\delta_0$ if $(\bfx,\bfv)$ solves \eqref{eq:ODE} with initial datum $(\bfx^0,\bfv^0)$ such that $\|\bfv^0\|\leq M$ and $(\tbx,\te,\tbw)$ is defined by 
\begin{align*}
\tbx^n&:=\bfx(n\,\Delta t)\,,&
\te^n&:=\frac12\|\bfv(n\,\Delta t)\|^2\,,&
\tbw^n&:=\bfv(n\,\Delta t)\,,&
\end{align*}
and $(\bfx,e,\bfw)$ solves \eqref{scheme:CN} with initial datum $(\bfx^0,e^0,\bfw^0)$ such that $\bfw^0=\bfv^0$, $e^0=\tfrac12\|\bfv^0\|^2$, 
then for any $n$ such that $0\leq n\Delta t\leq T$,
\[
\|(\GC^n,\eGC^n)-(\tGC^n,\teGC^n)\|
\,+\,\eps\,\|(\bfx^n,e^n)-(\tbx^n,\te^n)\|
\,+\,\eps^2\,\|\bfw^n-\tbw^n\|\\
\leq C\,\dfrac{(\Delta t)^2}{\eps^4}\,.
\]
\end{theorem}

\begin{proof}
Note that Section~\ref{s:continuous} and Proposition~\ref{p:unif} provide the bounds necessary to apply Proposition~\ref{p:stab}. Therefore by combining Proposition~\ref{p:stab} with Lemma~\ref{le:cont-cons} and Corollary~\ref{c:cons} provides the theorem under the extra assumption (needed to apply Proposition~\ref{p:stab}) that $0<\eps\leq \eps_0'$ for some sufficiently small $\eps_0'$.

To complete the proof it is sufficient to observe that when $\eps_0'>0$ is held fixed and $\eps_0'\leq\eps\leq \eps_0$ standard direct convergence analysis in original variables $(\bfx,e,\bfw)$ provides the missing bounds (with constants blowing up dramatically in the limit $\eps_0'\to0$).
\end{proof}


\section{Proof of the main theorem, Theorem~\ref{th:main}}\label{s:main}

In this short section we explain how to combine the previous elements so as to conclude the proof of Theorem~\ref{th:main}. Half of the conclusion is the bound 
\[
\left\|(\bfx^n,e^n)-\left(\bfx,\tfrac12\|\bfv\|^2\right)(n\,\Delta t)\right\|
\,+\,\eps\,\|\bfw^n-\bfv(n\,\Delta t)\|\\
\leq C\,\dfrac{(\Delta t)^2}{\eps^5}\,,
\]
which is a direct consequence of Theorem~\ref{th:conv}.

The second half is 
\[
\left\|(\bfx^n,e^n)-\left(\bfx,\tfrac12\|\bfv\|^2\right)(n\,\Delta t)\right\|
\,+\,\eps\,\|\bfw^n-\bfv(n\,\Delta t)\|\\
\leq C\,\left(\,\eps+(\Delta t)^2\,\right)\,.
\]
It arises from the triangle inequality and four distinct inequalities,
\begin{align*}
\left\|(\bfx^n,e^n)-(\bfy^n,g^n)\right\|&\leq C\,\eps\,,&
\left\|\left(\bfx,\tfrac12\|\bfv\|^2\right)(n\,\Delta t)-(\bfy,g)(n\,\Delta t)\right\|&\leq C\,\eps\,,\\[0.5em]
\eps\,\|\bfw^n-\bfv(n\,\Delta t)\|&\leq C\,\eps\,,&
\left\|(\bfy^n,g^n)-(\bfy,g)(n\,\Delta t)\right\|&\leq C\,(\Delta t)^2\,.
\end{align*}
The first inequality stems from Proposition~\ref{p:disc-x}, the second one from Proposition~\ref{p:cont-x}, the third one from uniform bounds on $(\bfw^n)$ and $\bfv$ (the former being part of Proposition~\ref{p:unif}) and the fourth is a direct convergence analysis for the Scheme~\eqref{scheme:CN} towards System~\eqref{eq:asymp} that we omit as it is a non-stiff simpler version of the arguments detailed to prove Theorem~\ref{th:conv}.

\appendix

\section{Implicit solving}\label{s:implicit}

The asymptotic and convergence analyses of \eqref{scheme:CN} are carried out independently of knowing whether the implicit schemes \eqref{scheme:CN} and \eqref{scheme:asymp} do possess unique solutions. Our present goal is to briefly check that this may be enforced by taking $\Delta t$ small (independently of $\eps$). 

\begin{proposition}\label{p:asymp-implicit}
Let $T>0$, $M>0$. There exist $\delta_0>0$ such that when $0<\Delta t\leq \delta_0$ if $(\bfy,g)$ satisfies $|g|\leq \tfrac12M^2$ then there exists a unique $(\bfy',g')$ such that
\begin{equation} \label{scheme:asymp-implicit}
\begin{cases}
        & \dfrac{\bfy'- \bfy}{\Delta t} \,=\, - \dfrac{\bE^{\perp}}{b}\left(\dfrac{\bfy'+\bfy}{2}\right)
        \,-\,\dfrac{g'+g}{2}\nabla_{\bfx}^{\perp}\left(\dfrac1b\right)\left(\dfrac{\bfy'+\bfy}{2}\right)\,, \\
        & \dfrac{g'- g}{\Delta t} \, = \,-\dfrac{\phi(\bfy') - \phi(\bfy)}{\Delta t} \,.
\end{cases}
\end{equation}
\end{proposition}

\begin{proof}
The second equation provides $g'$ as a function of $(\bfy,g,\bfy')$ that may be inserted in the first equation so as to reduce solving System~\ref{scheme:asymp-implicit} to finding a solution $\bfy'$ to a single scalar equation, explicitly written as
\[
\bfy'\,=\,\bfy-\Delta t\,\left(\dfrac{\bE^{\perp}}{b}\left(\dfrac{\bfy'+\bfy}{2}\right)
\,-\,\left(g-\dfrac{\phi(\bfy') - \phi(\bfy)}{2}\right)\nabla_{\bfx}^{\perp}\left(\dfrac1b\right)\left(\dfrac{\bfy'+\bfy}{2}\right)
\right)\,.
\]
With respect to the $\bfy'$ variable the right-hand side of the foregoing equation has Lipschitz constant $\cO(\Delta t)$ (when $g$ is constrained to vary in a compact set). Thus the proposition follows from the Banach fixed-point theorem.
\end{proof}

\begin{proposition}\label{p:implicit}
Let $T>0$, $M>0$, $\eps_0>0$. There exist $\delta_0>0$ such that when $0<\eps\leq\eps_0$ and $0<\Delta t\leq \delta_0$ if $(\bfx,e,\bfw)$ satisfies $\|\bfw\|\leq M$, $|e|\leq \tfrac12M^2$ then there exists a unique $(\bfx',e',\bfw')$ such that
\begin{equation} \label{scheme:CN-implicit}
\begin{cases}
        & \dfrac{\bfx'- \bfx}{\Delta t} \, =\,\dfrac{\bfw'+\bfw}{2\,\eps}
        \,-\,\left(\dfrac{e'+e}{2}\,-\, \dfrac{1}{2} \left\|\dfrac{\bfw'+\bfw}{2}\right\|^{2}\right)\nabla_{\bfx}^{\perp}\left(\dfrac1b\right)\left(\dfrac{\bfx'+\bfx}{2}\right)\,, \\[0.5em]
        & \dfrac{e' - e}{\Delta t} \,= \,-\dfrac{\phi(\bfx') - \phi(\bfx)}{\Delta t} \,,\\[0.5em]
        & \dfrac{\bfw'- \bfw}{\Delta t} \,= \, \dfrac{1}{\eps} \bE\left(\dfrac{\bfx'+\bfx}{2}\right) \, - \, b\left(\dfrac{\bfx'+\bfx}{2}\right) \dfrac{(\bfw')^{\perp}+\bfw^{\perp}}{2\,\eps^{2}} \,. 
\end{cases}
\end{equation}
\end{proposition}

\begin{proof}
The strategy is to make the most of stability properties in Lemma~\ref{le:stab} so as to extend the proof of Proposition~\ref{p:asymp-implicit}.

To begin with we observe that one may replace $\bfw'$ with $\ubw:=(\bfw'+\bfw)/2$ since $\bfw'$ is then recovered through $\bfw'=\bfw+2\,(\ubw-\bfw)$. Before repeating the elements of the proof of Proposition~\ref{p:asymp-implicit} we need to derive an a priori bound on $\ubw$. With $\bJ$ as in \eqref{notation_J} (and Lemma~\ref{le:stab}) and $\lambda:=\Delta t/\eps^2$ the third equation of \eqref{scheme:CN-implicit} may be written as
\begin{equation}\label{eq:aux-implicit}
\ubw=\left(\Id+\frac{\lambda}{2}\,b\left(\dfrac{\bfx'+\bfx}{2}\right)\bJ\right)^{-1}\left(\bfw
+\frac{\eps\,\lambda}{2}\,\bE\left(\dfrac{\bfx'+\bfx}{2}\right)\right)
\end{equation}
so that it follows from Lemma~\ref{le:stab} that for some constant $K$ (depending on the lower bound on $b$) any solution $(\bfx',e',\bfw')$ to \eqref{scheme:CN-implicit} satisfies $\|\ubw\|\leq K\,(M+\eps_0\,\|\bE\|_{L^\infty})$. 

Then we note that the second equation of System~\ref{scheme:CN-implicit} provides $e'$ as a function of $(\bfx,e,\bfx')$ and that the outcome may be inserted in its first equation so as to reduce it to
\[
\bfx'\,=\,\bfx+\Delta t\,\left(
\dfrac{\ubw}{\eps}
\,-\,\left(e-\dfrac{\phi(\bfx')-\phi(\bfx)}{2}\,-\, \dfrac{1}{2} \left\|\ubw\right\|^{2}\right)\nabla_{\bfx}^{\perp}\left(\dfrac1b\right)\left(\dfrac{\bfx'+\bfx}{2}\right)
\right)\,.
\]
Heeding the a priori bound on $\ubw$, we observe that with respect to the $\bfx'$ variable the right-hand side of the foregoing equation has Lipschitz constant $\cO(\Delta t)$. Thus provided that $\Delta t$ is sufficiently small (independently of $\eps$) the Banach fixed-point theorem implies that this equation prescribes uniquely $\bfx'$ as a function of $(\bfx,e,\ubw)$. For latter use we note that with respect to the $\ubw$ variable the Lipschitz constant of the implicit map $(\bfx,e,\ubw)\mapsto \bfx'$ is of size $\cO(\Delta t/\eps)$. 

System \eqref{scheme:CN-implicit} is thus reduced to Equation~\eqref{eq:aux-implicit} where $\bfx'$ is replaced with the foregoing function of $(\bfx,e,\ubw)$. With this point of view Lemma~\ref{le:stab} implies that with respect to the $\ubz$ variable the right-hand side of Equation~\eqref{eq:aux-implicit} has Lipschitz constant $\cO(\Delta t)$. Therefore applying the Banach fixed-point theorem again concludes the proof. 
\end{proof}

Note that Propositions~\ref{p:asymp-implicit} and~\ref{p:implicit} deal with the solving of a single step of respective schemes. Nevertheless they may be combined with uniform estimates, as provided by Proposition~\ref{p:unif}, to be turned into verifications of the solvability property until times $0\leq n\Delta t\leq T$ provided that $\Delta t$ is sufficiently small (depending on $T$ but not on $\eps$).

\bibliographystyle{abbrv}
\bibliography{refer}

\end{document}